\documentclass[10pt]{article}
\usepackage[affil-it]{authblk}
\usepackage{amsmath}
\usepackage{amssymb}
\usepackage{amsthm}
\usepackage{amsfonts}
\usepackage{mathrsfs}
\usepackage{amssymb, color}
\numberwithin{equation}{section}
\usepackage[top=1.25in, left=1.2in, right=1.2in, bottom=1.5in]{geometry}

\newtheorem{theorem}{Theorem}[section]
\newtheorem{proposition}[theorem]{Proposition}
\newtheorem{lemma}[theorem]{Lemma}

\newtheorem*{thm1}{Theorem 1}

\newtheorem*{cor}{Corollary 1}

\newcommand{\cl}{\text{curl}}
\newcommand{\eps}{\epsilon}
\newcommand{\mb}{\mathbb}
\newcommand{\haus}{\mathcal{H}}
\newcommand{\lt}{\left}
\newcommand{\rt}{\right}
\newcommand{\R}{\mathbb{R}}
\newcommand{\blue}[1]{{\textcolor{black}{#1}}}

\AtBeginDocument{

}

\title{Convergence of the Lawrence-Doniach Energy for Layered Superconductors with Magnetic Fields near $H_{c_1}$}
\author{Guanying Peng\footnote{Department of Mathematical Sciences, University of Cincinnati, Cincinnati, OH 45221, USA. E-mail: penggg@ucmail.uc.edu.}}

\date{}

\begin{document}

\maketitle

\begin{abstract}
	We analyze the Lawrence-Doniach model for three-dimensional highly anisotropic superconductors with layered structure. For such a superconductor occupying a bounded generalized cylinder in $\mathbb{R}^3$ with equally spaced parallel layers, we assume an applied magnetic field that is perpendicular to the layers with intensity $h_{ex}\sim|\ln\eps|$ as $\epsilon\rightarrow 0$, where $\epsilon$ is the reciprocal of the Ginzburg-Landau parameter. We prove Gamma-convergence of the Lawrence-Doniach energy as $\epsilon$ and the interlayer distance $s$ tend to zero, under the additional assumption that the layers are weakly coupled (i.e., $s\gg\eps$).
\end{abstract}


\section{Introduction}

This paper is devoted to the analysis of the Lawrence-Doniach model (with energy given by \eqref{LD}) for three-dimensional highly anisotropic superconductors with layered structure. Such discrete structure is common in high temperature superconductors (e.g., the cuprates). Because of the discrete layered structure, these superconductors exhibit very different material properties than isotropic superconductors, which can be well described by the celebrated Ginzburg-Landau model. (See the survey \cite{Iye} for a physical discussion on layered superconductors.) The Lawrence-Doniach model was proposed by Lawrence and Doniach \cite{LD} in 1971 as an alternative model to account for the anisotropy in layered superconductors. Unlike the Ginzburg-Landau model, which describes a superconductor as a continuous three-dimensional solid, the Lawrence-Doniach model treats the superconducting material as a stack of parallel superconducting layers with nonlinear Josephson coupling between them. It is generally considered a more complete theory for layered superconductors than other models (e.g., the anisotropic Ginzburg-Landau model).

\subsection{The energy functional}

We first recall the two-dimensional Ginzburg-Landau energy with magnetic fields. Let $\Omega\subset\mb R^2$ be a bounded simply connected smooth domain. Then the two-dimensional Ginzburg-Landau energy with magnetic fields is given by
\begin{equation*}
F_{\epsilon}(u,A)=\frac{1}{2}\int_{\Omega}\left[|\nabla_{A}u|^2+\frac{\lt(1-|u|^2\rt)^2}{2\epsilon^2}\right]dx+\frac{1}{2}\int_{\Omega}\lt(\text{curl}A-h_{ex}\rt)^2dx.
\end{equation*}
Here the first unknown $u:\Omega\rightarrow\mb C$ is the complex valued order parameter, whose modulus, $|u(x)|$, represents the density of superconducting electron pairs at the point $x$. For a minimizer of $F_{\eps}$, $|u(x)| \sim 1$ corresponds to a superconducting state at $x$, whereas $|u(x)|=0$ corresponds to a normal (nonsuperconducting) state at $x$. The second unknown $A=(A_1,A_2):\Omega\rightarrow\mb R^2$ is the magnetic potential, whose two-dimensional curl, $\cl A=\partial_1 A_2-\partial_2 A_1$, represents the induced magnetic field. The parameter $\epsilon>0$ is the reciprocal of the Ginzburg-Landau parameter $\kappa$, and $h_{ex}>0$ represents the strength of the applied magnetic field. We set $\nabla_{A}u=\nabla u-\imath Au$ on $\Omega$. The above energy $F_{\eps}$ represents the free energy of a cross-section of an infinitely long cylinder shaped superconductor subject to an applied magnetic field of intensity $h_{ex}$ that is perpendicular to the cross-section.

In this paper, we focus on the Lawrence-Doniach model. For some fixed $L>0$ and a bounded simply connected smooth domain $\Omega\subset\mb R^2$, we denote $D=\Omega\times(0,L)$, which is the bounded open cylinder in $\mb R^3$ with cross-section $\Omega$ and height $L$. We consider a layered superconductor occupying $\overline{D}$ with $N+1$ equally spaced layers of material occupying $\Omega_n=\Omega\times\{ns\}$, where $s=\frac{L}{N}$ is the interlayer distance. Assuming an applied magnetic field $h_{ex} \vec{e}_3$ which is perpendicular to the layers, the Lawrence-Doniach energy is given by
\begin{equation}\label{LD}
\begin{split}
\mathcal{G}_{LD}^{\epsilon, s}(\{u_n\}_{n=0}^N, \vec{A})&= s\sum^N_{n=0} \int_\Omega\left[\frac{1}{2}|\hat{\nabla}_{\hat{A}_{n}}u_n|^2+\frac{(1-|u_n|^2)^2}{4\epsilon^2}\right]d\hat{x}\\
&+s\sum^{N-1}_{n=0}\int_\Omega\frac{1}{2\lambda^{2}s^2}\lt|u_{n+1}-u_{n}e^{\imath \int_{ns}^{(n+1)s}A^{3}dx_{3}}\rt|^{2}d\hat{x}\\
&+\frac{1}{2}\int_{\mathbb{R}^3}\lt|\nabla\times{\vec{A}}-h_{ex}\vec{e}_3\rt|^{2}dx
\end{split}
\end{equation}
for $(\{u_n\}_{n=0}^N, \vec{A})$ such that
\begin{equation*}
\begin{cases}
&\{u_n\}_{n=0}^N\in [H^{1}(\Omega;\mathbb{C})]^{N+1} \quad\text{ and} \\
& \vec{A}\in E:=\{\vec{C}\in H^{1}_{loc}(\mathbb{R}^3;\mathbb{R}^3):(\nabla\times\vec{C})-h_{ex}\vec{e}_{3}\in L^{2}(\mathbb{R}^3;\mathbb{R}^3)\}.
\end{cases}
\end{equation*}
Similar to the Ginzburg-Landau energy, the first unknown $u_n:\Omega\rightarrow\mb C$ is the complex valued order parameter on the $n$th layer, and the second unknown $\vec{A}=(A^1,A^2,A^3):\mb R^3\rightarrow \mb R^3$ is the magnetic potential, whose three-dimensional curl, $\nabla\times\vec{A} = (\partial_2 A^3-\partial_3 A^2,\partial_3 A^1-\partial_1 A^3,\partial_1 A^2-\partial_2 A^1)$, is the induced magnetic field. The material parameter $\lambda >0$ represents the Josephson penetration depth, which is assumed to be fixed in this study. Throughout, we use $(\hat{\cdot})$ to denote two-dimensional vectors and operators. For example, we denote $\hat{x}=(x_1,x_2)$, $\hat{\nabla}=(\partial_{1},\partial_{2})$, $\hat{A}=(A^1,A^2)$ and $\hat{A}_n(\hat{x})=(A^{1}(\hat{x},ns),A^{2}(\hat{x},ns))$, the trace of $\hat{A}$ on the $n$th layer.

For the above models with magnetic fields, the behavior of energy minimizers is largely determined by the values of $h_{ex}$ versus $\eps$. Namely, there are two critical values of $h_{ex}$, denoted by $H_{c_1}\sim|\ln\epsilon|$ and $H_{c_3}\sim\frac{1}{\epsilon^2}$, at which the superconductor undergoes phase transitions from the superconducting state to the mixed state (coexistence of superconducting and normal states), and from the mixed state to the normal state, respectively. One of the central questions is to understand the vortex structure for minimizers with the strength of the magnetic field in different regimes. (A \emph{vortex} is an isolated zero of the order parameter $u$, around which $u$ has a nonzero winding number, called the \emph{degree of the vortex}.) For the underlying two-dimensional Ginzburg-Landau energy, the behavior of energy minimizers and their vortex structure in an applied magnetic field with modulus $h_{ex}$ in different regimes (e.g., $h_{ex}\sim|\ln\epsilon|,  |\ln\epsilon|\ll h_{ex}\ll\epsilon^{-2}$,  or $h_{ex}\geq\frac{C}{\epsilon^2}$) are now well understood. (See the book \cite{SS} and the references therein, and also \cite{JS} and \cite{Giorgi-Phillips}.) Recently, $\Gamma$-convergence results for the three-dimensional Ginzburg-Landau model in different energy regimes were obtained by Baldo et al. \cite{BJOS1}. For the Lawrence-Doniach energy, an analysis of minimizers for $h_{ex}$ in the first two regimes has been done by Alama et al. \cite{ABS} under certain periodicity assumptions. They also studied the cases when the magnetic fields are parallel to the layers or oblique in \cite{ABS} and \cite{ABS2}. Without the additional periodicity assumptions, Bauman and the author \cite{BP} proved an asymptotic formula for the minimum Lawrence-Doniach energy with $|\ln\epsilon|\ll h_{ex}\ll\epsilon^{-2}$ in the limiting case as $(\eps,s)\rightarrow (0,0)$, and obtained vortex structure information in this case. In the last regime, $h_{ex}\geq\frac{C}{\epsilon^2}$, it was shown by Bauman-Ko \cite{BK} that if $C$ is sufficiently large, all minimizers of the Lawrence-Doniach energy are in the normal phase. A similar result is known for the two-dimensional Ginzburg-Landau energy. (See \cite{Giorgi-Phillips}.) The goal of this paper is to investigate the limiting behavior of the Lawrence-Doniach energy with the intensity of the magnetic field in the regime $h_{ex}\sim|\ln\epsilon|$ without the additional periodicity assumptions.

Since $\vec{A}\in H^{1}_{loc}(\mathbb{R}^3;\mathbb{R}^3)$, by the trace theorem and the Sobolev embedding theorem, we have $\hat{A}_n\in H^{\frac{1}{2}}_{loc}(\mathbb{R}^2;\mathbb{R}^2) \subset L^4_{loc}(\mathbb{R}^2;\mathbb{R}^2)$. Therefore, the Lawrence-Doniach energy $\mathcal{G}_{LD}^{\epsilon, s}(\{u_n\}_{n=0}^N, \vec{A})$ is well-defined and finite. The existence of minimizers in the admissible space $ [H^{1}(\Omega;\mathbb{C})]^{N+1} \times E$ was shown by Chapman et al. \cite{CDG}. The minimizer satisfies the Euler-Lagrange equations associated to the Lawrence-Doniach energy given by
\begin{equation}\label{LDsystem}
\begin{cases}
(\hat{\nabla}-\imath\hat{A}_n)^{2}u_n+\frac{1}{\epsilon^2}(1-|u_n|^2)u_n+P_n=0 & \text{ on } \Omega,\\
\nabla\times(\nabla\times\vec{A})=(j_1,j_2,j_3) & \text{ in } \mathbb{R}^3,\\
(\hat{\nabla}-\imath\hat{A}_n)u_n\cdot\vec{n}=0 & \text{ on } \partial\Omega,\\
\nabla\times\vec{A}-h_{ex}\vec{e}_3\in L^2(\mathbb{R}^3;\mathbb{R}^3)
\end{cases}
\end{equation}
for all $n=0,1,\dddot\ ,N$, where
\begin{equation*}
P_n=\begin{cases}
\frac{1}{\lambda^{2}s^2}(u_{1}\bar{\Upsilon}_{0}^{1}-u_0) & \text{if $n=0$,}\\
\frac{1}{\lambda^{2}s^2}(u_{n+1}\bar{\Upsilon}_{n}^{n+1}+u_{n-1}\Upsilon_{n-1}^n-2u_n) & \text{if $0<n<N$,}\\
\frac{1}{\lambda^{2}s^2}(u_{N-1}\Upsilon_{N-1}^{N}-u_N) & \text{if $n=N$,}
\end{cases}
\end{equation*}
\begin{equation*}
\Upsilon_{n}^{n+1}=e^{\imath\int_{ns}^{(n+1)s}A^{3}dx_3} \quad\text{ for $n=0,1,\dddot\ ,N-1$,}
\end{equation*}
and
\begin{equation}\label{j}
\begin{split}
&j_i=-s\sum\limits_{n=0}^{N}(\partial_{i}u_{n}-\imath A_n^iu_n,-\imath u_n)\chi_{\Omega}(x_1,x_2)dx_{1}dx_{2}\delta_{ns}(x_3) \quad\text{ for $i=1,2,$}\\
&j_3=s\sum\limits_{n=0}^{N-1}\dfrac{1}{\lambda^{2}s^2}(u_{n+1}-u_{n}\Upsilon_{n}^{n+1},\imath u_{n}\Upsilon_{n}^{n+1})\chi_{\Omega}(x_1,x_2)\chi_{[ns,(n+1)s]}(x_3).
\end{split}
\end{equation}

We say that two configurations $(\{u_n\}_{n=0}^N, \vec{A})$ and $(\{v_n\}_{n=0}^N, \vec{B})$ in $[H^{1}(\Omega;\mathbb{C})]^{N+1}\times E$ are gauge-equivalent if there exists a function $g\in H^{2}_{loc}(\mathbb{R}^3)$ such that
\begin{equation}\label{gauge}
\begin{cases}
u_n(\hat{x})=v_n(\hat{x})e^{\imath g(\hat{x},ns)} &\text{ in } \Omega,\\
\vec{A}=\vec{B}+\nabla g &\text{ in } \mathbb{R}^3.\\
\end{cases}
\end{equation}
By direct calculations, one can check that $\mathcal{G}_{LD}^{\epsilon, s}$ (and each term in $\mathcal{G}_{LD}^{\epsilon, s}$) is invariant under the above gauge transformation, i.e., for two configurations $(\{u_n\}_{n=0}^N, \vec{A})$ and $(\{v_n\}_{n=0}^N, \vec{B})$ that are related by \eqref{gauge}, we have $\mathcal{G}_{LD}^{\epsilon, s}(\{u_n\}_{n=0}^N, \vec{A})=\mathcal{G}_{LD}^{\epsilon, s}(\{v_n\}_{n=0}^N, \vec{B})$. Let $\vec{a}=\vec{a}(x)$ be any fixed smooth \blue{divergence-free} vector field on $\mathbb{R}^3$ such that $\nabla\times\vec{a}=\vec{e}_3$ in
$\mathbb{R}^3$. Define the space $\check{H}^1(\mathbb{R}^3;\mb R^3)$ to be the completion of $C_0^{\infty}(\mathbb{R}^3;\mathbb{R}^3)$ with respect to the seminorm
\begin{equation*}
\lVert\vec{C}\rVert_{\check{H}^1(\mathbb{R}^3;\mb R^3)}=(\int_{\mathbb{R}^3}|\nabla\vec{C}|^2dx)^{\frac{1}{2}}.
\end{equation*}
From \cite{BK}, each $\vec{C}\in\check{H}^1(\mathbb{R}^3;\mb R^3)$ has a representative in $L^6(\mathbb{R}^3;\mathbb{R}^3)$ such that
\begin{equation}\label{H11}
\lVert\vec{C}\rVert_{L^6(\mathbb{R}^3;\mathbb{R}^3)}\leq 2\lVert\vec{C}\rVert_{\check{H}^1(\mathbb{R}^3;\mb R^3)},
\end{equation}
and
\begin{equation}\label{H12}
\lVert\vec{C}\rVert_{\check{H}^1(\mathbb{R}^3;\mb R^3)}^2=\int_{\mathbb{R}^3}(|\nabla\cdot\vec{C}|^2+|\nabla\times\vec{C}|^2)dx.
\end{equation}
Define the space $K$ to be
\begin{equation*}
K:=\{\vec{C}\in E:\nabla\cdot\vec{C}=0 \text{ and } \vec{C}-h_{ex}\vec{a}\in \check{H}^1(\mathbb{R}^3;
\mb R^3)\cap L^6(\mathbb{R}^3;\mathbb{R}^3)\}.
\end{equation*} 
It was proved in \cite{BK} that every pair $(\{u_n\}_{n=0}^N, \vec{A})\in[H^1(\Omega;\mathbb{C})]^{N+1}\times E$ is gauge-equivalent to another pair $(\{v_n\}_{n=0}^N, \vec{B}) \in [H^1(\Omega;\mathbb{C})]^{N+1}\times K$. In particular, any minimizer of $\mathcal{G}_{LD}^{\epsilon,s}$ in the admissible space $[H^1(\Omega;\mathbb{C})]^{N+1}\times E$ is gauge-equivalent to a minimizer in the space $[H^1(\Omega;\mathbb{C})]^{N+1}\times K$, called the ``Coulomb gauge" for $\mathcal{G}_{LD}^{\epsilon,s}$. It was also shown in \cite{BK} that a minimizer $(\{u_n\}_{n=0}^N, \vec{A})$ of $\mathcal{G}_{LD}^{\epsilon,s}$ satisfies $|u_n|\leq 1$ a.e. in $\Omega$, and that a minimizer in the Coulomb gauge satisfies $u_n \in C^{\infty}(\Omega;\mb C)$ and $\hat{A}_n \in H^1_{\text loc}(\mathbb{R}^2;\mb R^2)$ for all $n=0,1,\dddot\ ,N$. 

\subsection{Statement of the main results}

To state our main results, let us introduce some notations. Following notations in \cite{JS1}, for a function $u\in H^1(\Omega;\mathbb{C})$, we define the \emph{current} $j(u)$ and the \emph{Jacobian} $J(u)$ to be
\begin{equation*}
j(u)=(\imath u,\hat{\nabla}u),\quad J(u)=\frac{1}{2}\text{curl}j(u),
\end{equation*}
respectively. Here $(\imath u,\hat{\nabla}u)$ is a real vector in $\mathbb{R}^2$ with components $(\imath u,\partial_j u)$ for $j=1,2$, where $(a+\imath b,c+\imath d)=ac+bd$ for two complex numbers $a+\imath b$ and $c+\imath d$. For the two-dimensional Ginzburg-Landau energy, the Jacobian carries important topological information of vortices (e.g., degree and location). For a configuration $(\{u_n\}_{n=0}^N,\vec A)\in[H^1(\Omega;\mathbb{C})]^{N+1}\times E$, we define the discrete version of the current and the Jacobian as
\begin{equation}\label{ja}
j^{\epsilon,s}(\{u_n\}_{n=0}^N)=\sum\limits_{n=0}^{N-1}j(u_n)\chi_n(x_3), \quad J^{\epsilon,s}(\{u_n\}_{n=0}^N)=\sum\limits_{n=0}^{N-1}J(u_n)\chi_n(x_3),
\end{equation}
respectively, where
\begin{equation}\label{chi}
\chi_n(x_3) = 
\begin{cases}
\chi_{(0,s)}(x_3) & \text{ for } n=0,\\
\chi_{[ns,(n+1)s)}(x_3) & \text{ for } n=1,\dddot\ ,N-1.
\end{cases}
\end{equation}
Given some constant $h_0\geq 0$, we define spaces $E_0$ and $K_0$ that are parallel to $E$ and $K$. More precisely, let
\begin{equation*}
E_0:=\{\vec{C}\in H^{1}_{loc}(\mathbb{R}^3;\mathbb{R}^3):(\nabla\times\vec{C})-h_{0}\vec{e}_{3}\in L^{2}(\mathbb{R}^3;\mathbb{R}^3)\},
\end{equation*}
and
\begin{equation}\label{K0}
K_0:=\{\vec{C}\in E_0:\nabla\cdot\vec{C}=0 \text{ and } \vec{C}-h_{0}\vec{a}\in \check{H}^1(\mathbb{R}^3;\mb R^3)\cap L^6(\mathbb{R}^3;\mathbb{R}^3)\}.
\end{equation}
Let $\mathcal{M}(D)$ be the space of finite Radon measures. Also we define the space
\begin{equation}\label{V}
V:=\{v\in L^2(D;\mb R^2): \cl v \in \mathcal{M}(D)\}.
\end{equation}
For a pair $(v,\vec A)\in L^2(D;\mb R^2)\times E_0$, we define the energy functional
\begin{equation}\label{G0}
\mathcal{G}_{h_0}(v,\vec A):=\frac{1}{2}\left[\parallel v-\hat A \parallel_{L^2(D)}^2 + |\text{curl}v|(D) + \parallel \nabla\times\vec A - h_0\vec e_3\parallel_{L^2(\mathbb R^3)}^2\right],
\end{equation}
where $|\text{curl}v|(D)$ denotes the total variation of the Radon measure $\text{curl}v$, with the convention to understand $\mathcal{G}_{h_0}(v,\vec A)$ equal to $+\infty$ if $\text{curl}v\notin\mathcal{M}(D)$. Our main result is the following $\Gamma$-convergence of the Lawrence-Doniach energy:

\begin{thm1}\label{T1}
	Assume $\lim_{\epsilon\rightarrow 0}\frac{h_{ex}}{|\ln\epsilon|}=h_0$ for some $0\leq h_0<\infty$ and
	\begin{equation}\label{h}
	s|\ln\epsilon|\rightarrow\infty \quad \text{ as } (\epsilon,s)\rightarrow (0,0).
	\end{equation}
	
	\noindent\emph{(Compactness and lower bound)} For any sequence $(\{u^{\epsilon}_n\},\vec{A}^{\epsilon,s})\in [H^1(\Omega;\mathbb{C})]^{N+1}\times K$ such that
	\begin{equation}\label{thm1.5}
	\mathcal{G}_{LD}^{\epsilon,s}(\{u_n^{\eps}\},\vec{A}^{\epsilon,s}) \leq C_0|\ln\epsilon|^2,
	\end{equation}
	where $C_0$ is a constant independent of $\epsilon$ and $s$, we have, up to a subsequence as $(\epsilon,s) \rightarrow (0,0)$, 
	\begin{equation}\label{thm1.1}
	\sum\limits_{n=0}^{N-1}\frac{j(u_n^{\eps})}{|u_n^{\eps}||\ln\eps|}\chi_n \rightharpoonup v \text{ in } L^2(D;\mathbb{R}^2), \quad \frac{j^{\epsilon,s}}{|\ln\epsilon|}\rightharpoonup v \text{ in } L^{\frac{4}{3}}(D;\mathbb{R}^2),
	\end{equation}
	\begin{equation}\label{thm1.01}
	\frac{J^{\epsilon,s}}{|\ln\epsilon|}\rightharpoonup w \text{ in } (C^{0,\alpha}_c(D))^{*} \text{ for all } 0<\alpha\leq 1,
	\end{equation}
	and
	\begin{equation}\label{thm1.2}
	\frac{\vec A^{\epsilon,s} - h_{ex}\vec a}{|\ln\epsilon|} \rightharpoonup \vec A-h_0\vec a \text{\quad in } \check H^1(\mathbb R^3;\mathbb R^3)
	\end{equation}
	for some $v\in L^2(D;\mathbb{R}^2)$ such that $w=\frac{1}{2}\mathrm{curl}v$ is a Radon measure, and for some $\vec A\in h_0\vec a+\check{H}^1(\mathbb{R}^3;\mathbb{R}^3)$. In addition, we have
	\begin{equation}\label{thm1.3}
	\liminf_{(\epsilon,s)\rightarrow(0,0)}\frac{\mathcal{G}_{LD}^{\epsilon, s}\blue{(\{u_n^{\eps}\},\vec{A}^{\epsilon,s})}}{|\ln\epsilon|^2} \geq \mathcal{G}_{h_0}(v,\vec A).
	\end{equation}
	
	\noindent\emph{(Upper bound)} Given $(v,\vec A)\in V\times K_0$, there exists a sequence $(\{\blue{\tilde{u}_n^{\epsilon}}\},\blue{\tilde{\vec{A}}^{\epsilon}})$ satisfying the compactness results \eqref{thm1.1}-\eqref{thm1.2}. Furthermore, we have
	\begin{equation}\label{thm1.4}
	\limsup_{(\epsilon,s)\rightarrow(0,0)}\frac{\mathcal{G}_{LD}^{\epsilon, s}(\{\blue{\tilde{u}_n^{\epsilon}}\},\blue{\tilde{\vec{A}}^{\epsilon}})}{|\ln\epsilon|^2} \leq \mathcal{G}_{h_0}(v,\vec A).
	\end{equation}
\end{thm1}

\blue{Note that the constructed magnetic potential $\tilde{\vec{A}}^{\eps}$ in the recovery sequence does not depend on $s$.} The regime for $h_{ex}$ under consideration is a lot more subtle than the higher regime studied in \cite{BP}, since for the regime $h_{ex}\sim|\ln\eps|$, the superconductor undergoes a phase transition from the superconducting state to the mixed state. \blue{The same regime for the Lawrence-Doniach model was considered in the work of Alama et al. \cite{ABS}, in which the energy is minimized among configurations whose gauge-invariant quantities are periodic with respect to a given parallelepiped.} In that case, the periodicity assumptions simplify the problem significantly. Namely, it was proved that, for a minimizer of the gauge periodic problem, the order parameters $u_n$ are all equal and $A^3$ is identically zero. In particular, the Josephson coupling term
\begin{equation}\label{JC}
s\sum^{N-1}_{n=0}\int_\Omega\frac{1}{2\lambda^{2}s^2}\lt|u_{n+1}-u_{n}e^{\imath \int_{ns}^{(n+1)s}A^{3}dx_{3}}\rt|^{2}d\hat{x}
\end{equation}
vanishes. As a result, the Lawrence-Doniach energy reduces to a sum of copies of the two-dimensional Ginzburg-Landau energies on the layers. \blue{In our case} without the periodicity \blue{and energy minimizer} assumptions, such dimension reduction techniques do not work, \blue{although the assumption (\ref{h}) makes the treatment of the coupling term (\ref{JC}) very easy.} Note that this assumption is only used to guarantee that the energy from (\ref{JC}) with a scaling factor $|\ln\eps|^2$ converges to zero in the limit. 

Theorem \ref{T1} extends the $\Gamma$-convergence results on the two-dimensional (see \cite{SS1} and \cite{JS}) and three-dimensional (see \cite{BJOS1}) Ginzburg-Landau models to the Lawrence-Doniach model. Our results are parallel to those on the former. Here, the $\Gamma$-limit $\mathcal{G}_{h_0}$ defined in \eqref{G0} differs from those for the Ginzburg-Landau models, in that $\mathcal{G}_{h_0}$ includes both two-dimensional (terms involving $v$) and three-dimensional (terms involving $\vec A$) functions. In particular, the limiting Jacobian, $\cl v$, is a scalar valued measure and does not include the $x_3$ derivative of $v$. Therefore, our problem has features of both the two-dimensional and three-dimensional Ginzburg-Landau models. Nevertheless, the Lawrence-Doniach model shares more features with the three-dimensional Ginzburg-Landau model. Indeed, for the Ginzburg-Landau models with magnetic fields in the regime considered here, through a convex duality argument, one can rewrite the minimization of the limiting energy functional in the form of a constrained minimization problem of obstacle type (see \cite{SS1}, \cite{JS} and \cite{BJOS2}). A major difference between the two-dimensional and three-dimensional cases is that, in the latter case, the obstacle becomes nonlocal. This makes the analysis of the three-dimensional obstacle type problem much more challenging than its two-dimensional counterpart. (See \cite{BJOS2} for more details.) Our problem of minimization of the limiting functional $\mathcal{G}_{h_0}$ also corresponds to a nonlocal obstacle problem. \blue{It would be interesting to further investigate the critical value $h_{*}$ for $h_0$ in the limiting functional $\mathcal{G}_{h_0}$, below which minimizers of $\mathcal{G}_{h_0}$ satisfy $\text{curl}v=0$. The leading order of the first critical field $H_{c_1}$ of the Lawrence-Doniach energy is then given by $h_{*}|\ln\epsilon|$. The characterizations of $H_{c_1}$ for the three-dimensional Ginzburg-Landau model were obtained in \cite{ABM} when the domain is a ball, and in \cite{BJOS2} for general domains, through very different arguments.} A recent work of Athavale et al. \cite{AJNO} includes detailed discussions on the minimization of energy functionals involving total variation measures.

\blue{The extra assumption \eqref{h} in Theorem \ref{T1} makes the Lawrence-Doniach model in the extremely discrete scenario.} In some sense, \blue{this assumption} imposes a weak coupling condition between adjacent layers in the Lawrence-Doniach model. Nevertheless, the regime under consideration is an interesting regime. In fact, when $s\ll\epsilon$, it is expected that the Lawrence-Doniach model converges, in some sense, to an anisotropic Ginzburg-Landau model (see \cite{CDG} and \cite{BK}), which is not significantly different from the standard Ginzburg-Landau model in certain aspects. The assumption \eqref{h} is a key factor that leads to the mixture of two-dimensional and three-dimensional terms in the limiting functional $\mathcal{G}_{h_0}$, and consequently, that makes the problem more different from the three-dimensional Ginzburg-Landau model. It is not clear to the author whether or not the limiting functional $\mathcal{G}_{h_0}$ derived in Theorem \ref{T1} also serves as the $\Gamma$-limit of $\mathcal{G}_{LD}^{\epsilon, s}$ with $\eps$ and $s$ in other regimes, especially when $s\ll\epsilon$. 

The proof of the compactness and lower bound in Theorem \ref{T1} uses a standard slicing argument that has been used by Jerrard-Soner \cite{JS1} and Sandier-Serfaty \cite{SS2} for the Ginzburg-Landau functional in higher dimensions. A key point is to use small balls to cover the Jacobian in order to separate the energy contribution of the Jacobian from that of the other terms (see Theorem \ref{lem2.3}). 

The proof of the upper bound is a lot more involved. In Section 4, we construct the order parameters on the layers. Essentially, on each layer, we follow the construction of test functions for the two-dimensional Ginzburg-Landau energy used in \cite{SS1} and \cite{JS}. Here an extra level of subtlety comes from the limiting process as the interlayer distance $s$ approaches zero, which creates some extra technical difficulties. In Section 5, we construct the magnetic potential by slightly modifying that of the given configuration $(v,\vec{A})$.

As a consequence of Theorem \ref{T1}, we have the following \blue{compactness result for energy minimizers}:

\begin{cor}\label{cor1}
	Assume $\lim_{\epsilon\rightarrow 0}\frac{h_{ex}}{|\ln\epsilon|}=h_0$ for some $0\leq h_0<\infty$ and the hypothesis (\ref{h}). Let $(\{u_n^{\eps}\},\vec{A}^{\epsilon,s})\in[H^1(\Omega;\mathbb{C})]^{N+1}\times K$ be minimizers of $\mathcal{G}_{LD}^{\epsilon,s}$. We have, up to a subsequence as $(\eps,s)\rightarrow(0,0)$,
	\begin{equation*}
	\nabla\times\lt(\nabla\times\frac{\vec A^{\epsilon,s}}{|\ln\eps|}\rt)\rightharpoonup \lt((v_0^1-A_0^1)\chi_D,(v_0^2-A_0^2)\chi_D,0\rt) \quad\text{ in } \mathcal{M}(\mb R^3;\mb R^3),
	\end{equation*}
	where \blue{$(v_0,\vec A_0)$ is a minimizer of $\mathcal{G}_{h_0}$ and} $\chi_D$ is the characteristic function of the domain $D$.
\end{cor}

The above Corollary \ref{cor1} gives more compactness for the magnetic potential \blue{of energy minimizers} than has been obtained in Theorem \ref{T1} \blue{for general sequences with energy upper bound}. Note that, \blue{for minimizers of the Lawrence-Doniach energy,} $\nabla\times\lt(\nabla\times\vec A\rt)$ is a sum of singular measures supported on the layers, as can be seen from the Euler-Langrange equations \eqref{LDsystem}. As a result, the magnetic potential has less regularity than the order parameters. \blue{It was proved in \cite{BK} that $A^1$ and $A^2$ can be represented using sum of single layer potentials defined on the layers.} A major challenge in the analysis of the Lawrence-Doniach model comes from the discrete structure of the problem. In the limit as the interlayer distance $s$ tends to zero, we need to show compactness for discrete quantities in the form of the discrete current and Jacobian defined in \eqref{ja}. Such compactness results are more difficult to establish for the magnetic potential $\vec A$ for the reasons mentioned above. It is crucial to understand how the vector field $\vec A$ relates to its traces in this particular context. \blue{In \cite{BP}, some powerful a priori estimates for the magnetic potential were established based on the results in \cite{BK}. Those estimates turned out to be very useful for the analysis of the Lawrence-Doniach model.}

\subsection{Outline of the paper}

This paper is organized as follows: in Section 2, we provide some preliminaries that are needed in later sections. The proof of Theorem \ref{T1} constitutes the major part of this paper, and is included in Sections 3 through 6. We conclude Section 6 with the proof of Corollary \ref{cor1}. Finally, the last section is an appendix, which contains some approximation and extension results for the space $V$.

\section{Preliminaries}

In this section, we provide some preliminary results that are needed for later sections. First we prove existence of minimizers of the limiting functional $\mathcal{G}_{h_0}$ defined in \eqref{G0}. The proof is standard and follows from the direct method in the calculus of variations. However, since the functional $\mathcal{G}_{h_0}$ contains a mixture of two-dimensional and three-dimensional functions, we believe that it is worth including a proof here.

\begin{proposition}\label{prop2.1}
	The minimum of $\mathcal{G}_{h_0}$ over $V\times E_0$ and over $V\times K_0$ is achieved. Moreover, we have
	\begin{equation*}
	\min_{V\times E_0} \mathcal{G}_{h_0} = \min_{V\times K_0} \mathcal{G}_{h_0}.
	\end{equation*}
\end{proposition}

\begin{proof}
	Given $(v,\vec A)\in V\times E_0$, using a similar argument as in Lemma 2.1 of \cite{BK}, there exists a function $g\in H^2_{loc}(\mb R^3)$ such that $\vec A + \nabla g\in K_0$. By simple calculations, one can check that
	\begin{equation*}
	\mathcal{G}_{h_0}(v+\hat{\nabla}g,\vec A + \nabla g) = \mathcal{G}_{h_0}(v, \vec A),
	\end{equation*}
	where recall that $\hat{\nabla}$ is the gradient operator with respect to $x_1$ and $x_2$. Therefore, it suffices to show that the minimum of $\mathcal{G}_{h_0}$ is achieved by some $(v_0,\vec A_0)\in V\times K_0$.
	
	Let $\{(v_k, \vec A_k)\} \subset V \times K_0$ be a minimizing sequence of $\mathcal{G}_{h_0}$. By the definition of the space $K_0$ in \eqref{K0}, we have $\nabla\cdot \vec A_k=0$. Recall that \blue{$\vec a(x)$ is a smooth vector field on $\mathbb{R}^3$ such that $\nabla\cdot \vec a = 0$ and $\nabla\times\vec a=\vec{e}_3$}. Therefore, we have $\nabla\cdot(\vec A_k-h_0 \vec a) = 0$. By \eqref{H12}, we have
	\begin{equation*}
	\lVert\vec A_k - h_0 \vec a\rVert_{\check H^1(\mb R^3;\mb R^3)}^2 = \int_{\mathbb R^3} \left | \nabla\times\vec A_k - h_0 \vec e_3 \right |^2 dx.
	\end{equation*}
	Since $\{\mathcal{G}_{h_0}(v_k, \vec A_k)\}$ is a bounded sequence, it follows that $\{\vec A_k - h_0 \vec a\}$ forms a bounded sequence in $\check H^1(\mb R^3;\mb R^3)$, and, by \eqref{H11}, in $L^6(\mb R^3;\mb R^3)$. Consequently, there exists $\vec A_0$ such that $\vec A_0-h_0 \vec a\in\check H^1\cap L^6$ and, upon extraction, 
	\begin{equation}\label{prop4.1.1}
	\vec A_k - h_0 \vec a \rightharpoonup \vec A_0 - h_0 \vec a \quad\text{ in } \check H^1(\mb R^3;\mb R^3) \text{ and in } L^6(\mb R^3;\mb R^3).
	\end{equation}
	(Here we do not distinguish between the original sequence and its convergent subsequences.) It is then clear that $\nabla\cdot(\vec A_0-h_0\vec a)=0$. Therefore we have $\vec A_0\in K_0$. \blue{It follows from the boundedness of $\{\hat{A}_k\}$ in $L^6(D;\R^2)$ and (\ref{prop4.1.1})} that $\{\hat A_k\}$ is bounded in $L^2(D;\mb R^2)$ \blue{and}
	\begin{equation}
	\hat{A}_k \rightharpoonup \hat{A}_0 \quad \text{ in } L^2(D;\R^2).
	\end{equation}
	This along with the boundedness of $\{v_k-\hat A_k\}$ in $L^2(D;\mb R^2)$ implies that $\{v_k\}$ is bounded in $L^2(D;\mb R^2)$. So there exists $v_0 \in L^2(D;\mb R^2)$ such that, upon extraction,
	\begin{equation}\label{prop4.1.2}
	v_k \rightharpoonup v_0 \quad \text{ in } L^2(D;\mb R^2). 
	\end{equation}
	Now it only remains to show that $\text{curl} v_0$ is a finite Radon measure. To this end, we take a test function $\varphi \in C^1_c(D)$. It follows from \eqref{prop4.1.2} and an integration by parts that
	\begin{equation*}
	-\int_D v_0 \cdot\hat\nabla^{\perp}\varphi dx = -\lim_{k\rightarrow\infty} \int_D v_k \cdot\hat\nabla^{\perp}\varphi dx = \lim_{k\rightarrow\infty}\int_D \varphi d\cl v_k,
	\end{equation*}
	where $\hat{\nabla}^{\perp}$ denotes the operator $(-\partial_2,\partial_1)$. Since $\{\text{curl}v_k\}$ is bounded in $\mathcal M(D)$, we have 
	\begin{equation*}
	\int_D \varphi d\cl v_k\leq C\sup|\varphi|
	\end{equation*}
	for some constant $C$ independent of $k$ and $\varphi$. Therefore we deduce that 
	\begin{equation*}
	\int_{D}\varphi d\cl v_0 = -\int_D v_0 \cdot\hat\nabla^{\perp}\varphi dx \leq C\sup|\varphi|
	\end{equation*}
	for all $\varphi \in C^1_c(D)$, which implies $\text{curl}v_0 \in \mathcal M(D)$. By \eqref{prop4.1.2}, we have
	\begin{equation}
	\label{prop4.1.6}
	\text{curl}v_k \rightharpoonup \text{curl}v_0 \quad \text{ in } (C_c(D))^*.
	\end{equation}
	Putting \eqref{prop4.1.1}-\eqref{prop4.1.6} together, and using lower semicontinuity, we conclude that
	\begin{equation*}
	\mathcal{G}_{h_0}(v_0, \vec A_0) \leq \liminf \mathcal{G}_{h_0}(v_k, \vec A_k)=\inf_{V\times K_0} \mathcal{G}_{h_0}. 
	\end{equation*}
	Hence, $(v_0, \vec A_0)\in V\times K_0$ is a minimizer of $\mathcal{G}_{h_0}$.
\end{proof}

\blue{The following result will be used repeatedly.}

\begin{lemma}\label{l2.2}
	\blue{For every $\vec{B}\in H^1_{loc}(\mathbb{R}^3;\mathbb{R}^3)$, we have
		\begin{equation}\label{l2.2.0}
		\sum_{n=0}^{N-1}\int_{ns}^{(n+1)s}\int_{\Omega}\lt|\vec{B}-\vec{B}_n\rt|^2 d\hat x dx_3 \leq s^2 \int_{D}\lt|\nabla\vec{B}\rt|^2 dx,
		\end{equation}
		where $\vec{B}_n$ is the trace of $\vec{B}$ on $\Omega_n:=\Omega\times\{ns\}$. }
\end{lemma}

\begin{proof}
	First assume that $\vec{B}\in C^{\infty}(\mathbb{R}^3;\mathbb{R}^3)\cap H^1_{loc}(\mathbb{R}^3;\mathbb{R}^3)$. Using H\"{o}lder's inequality, we have
	\begin{equation*}
	\begin{split}
	&\int_{ns}^{(n+1)s}\int_{\Omega}\lt|\vec{B}(x)-\vec{B}_n(\hat x)\rt|^2 d\hat x dx_3 = \int_{ns}^{(n+1)s}\int_{\Omega}\lt|\int_{ns}^{x_3}\partial_3\vec{B}(\hat x,t) dt\rt|^2 d\hat x dx_3\\
	&\quad\quad\quad\quad\quad\quad\quad\quad\quad\leq \int_{ns}^{(n+1)s}\int_{\Omega}\lt(\int_{ns}^{(n+1)s}\lt|\partial_3\vec{B}(\hat x,t)\rt| dt\rt)^2 d\hat x dx_3\\
	&\quad\quad\quad\quad\quad\quad\quad\quad\quad\leq s\int_{ns}^{(n+1)s}\int_{\Omega}\lt(\int_{ns}^{(n+1)s}\lt|\partial_3\vec{B}(\hat x,t)\rt|^2 dt\rt) d\hat x dx_3\\
	&\quad\quad\quad\quad\quad\quad\quad\quad\quad = s^2 \int_{ns}^{(n+1)s}\int_{\Omega}\lt|\partial_3\vec{B}(\hat x,t)\rt|^2  d\hat x dt.
	\end{split}
	\end{equation*}
	It follows that
	\begin{equation}\label{l2.2.1}
	\begin{split}
	&\sum_{n=0}^{N-1}\int_{ns}^{(n+1)s}\int_{\Omega}\lt|\vec{B}-\vec{B}_n\rt|^2 d\hat x dx_3 \\
	&\quad\quad\quad\leq s^2 \sum_{n=0}^{N-1} \int_{ns}^{(n+1)s}\int_{\Omega}\lt|\partial_3\vec{B}(\hat x,t)\rt|^2  d\hat x dt\leq s^2 \int_{D}\lt|\nabla\vec{B}\rt|^2 dx.
	\end{split}
	\end{equation}
	
	Now assume that $\vec{B}\in H^1_{loc}(\mathbb{R}^3;\mathbb{R}^3)$. Let $\{\vec{B}^k\}\subset C^{\infty}(\mathbb{R}^3;\mathbb{R}^3)\cap H^1_{loc}(\mathbb{R}^3;\mathbb{R}^3)$ be a sequence such that $\vec{B}^k\rightarrow \vec{B}$ in $H^1_{loc}(\mathbb{R}^3;\mathbb{R}^3)$. Using Young's inequality yields
	\begin{equation}\label{l2.2.3}
	\begin{split}
	&\sum_{n=0}^{N-1}\int_{ns}^{(n+1)s}\int_{\Omega}\lt|\vec{B}-\vec{B}_n\rt|^2 d\hat x dx_3\\
	&\quad\quad\leq \sum_{n=0}^{N-1}\int_{ns}^{(n+1)s}\int_{\Omega}2\lt(1+\frac{1}{\sigma}\rt)\lt(\lt|\vec{B}-\vec{B}^k\rt|^2+\lt|\vec{B}^k_n-\vec{B}_n\rt|^2\rt)d\hat x dx_3\\
	& \quad\quad\quad+\sum_{n=0}^{N-1}\int_{ns}^{(n+1)s}\lt(1+\sigma\rt)\lt|\vec{B}^k-\vec{B}^k_n\rt|^2 d\hat x dx_3
	\end{split}
	\end{equation}
	for all $\sigma>0$. Since $\vec{B}^k\rightarrow \vec{B}$ in $H^1_{loc}(\mathbb{R}^3;\mathbb{R}^3)$, we have
	\begin{equation}\label{l2.2.4}
	\sum_{n=0}^{N-1}\int_{ns}^{(n+1)s}\int_{\Omega}\lt|\vec{B}-\vec{B}^k\rt|^2 d\hat x dx_3 = \int_{D}\lt|\vec{B}-\vec{B}^k\rt|^2dx\overset{k\rightarrow\infty}{\longrightarrow} 0.
	\end{equation}
	Using (\ref{l2.2.1}) we have
	\begin{equation}\label{l2.2.5}
	\sum_{n=0}^{N-1}\int_{ns}^{(n+1)s}\int_{\Omega}\lt|\vec{B}^k-\vec{B}^k_n\rt|^2 d\hat x dx_3 \leq s^2 \int_{D}\lt|\nabla\vec{B}^k\rt|^2 dx\overset{k\rightarrow\infty}{\longrightarrow} s^2 \int_{D}\lt|\nabla\vec{B}\rt|^2 dx.
	\end{equation}
	For each $n$, we identify $\Omega_n$ as a flat portion of some bounded smooth domain $\omega_n\subset \mathbb{R}^3$. By employing the trace theorem, we have
	\begin{equation}\label{l2.2.2}
	\lVert \vec{B}^k_n - \vec{B}_n \rVert_{L^2(\Omega_n)} \leq C \lVert \vec{B}^k - \vec{B} \rVert_{H^1(\omega_n)}
	\end{equation}
	for some constant $C$ depending only on $\omega_n$. One can choose $\omega_n$ to be $\omega_0+(0,0,ns)$, i.e., vertical translation of $\omega_0$ by $ns$, so that the constant $C$ in (\ref{l2.2.2}) is independent of $n$. Therefore, it follows from (\ref{l2.2.2}) that
	\begin{equation*}
	\begin{split}
	&\sum_{n=0}^{N-1}\int_{ns}^{(n+1)s}\int_{\Omega}\lt|\vec{B}^k_n-\vec{B}_n\rt|^2 d\hat x dx_3 \\
	&\quad\quad= s\sum_{n=0}^{N-1}\lVert \vec{B}^k_n - \vec{B}_n \rVert_{L^2(\Omega_n)}^2\leq sC\sum_{n=0}^{N-1}\lVert \vec{B}^k - \vec{B} \rVert_{H^1(\omega_n)}^2.
	\end{split}
	\end{equation*}
	Now we fix some bounded domain $\omega\subset\mathbb{R}^3$ sufficiently large such that $\cup_n \omega_n\subset \omega$. Note that $sN=L$ is the height of the cylinder $D$, which is fixed. Therefore,
	\begin{equation}\label{l2.2.6}
	\sum_{n=0}^{N-1}\int_{ns}^{(n+1)s}\int_{\Omega}\lt|\vec{B}^k_n-\vec{B}_n\rt|^2 d\hat x dx_3 \leq sN C\lVert \vec{B}^k - \vec{B} \rVert_{H^1(\omega)}^2\overset{k\rightarrow\infty}{\longrightarrow} 0.
	\end{equation}
	Finally, by putting (\ref{l2.2.4}), (\ref{l2.2.5}) and (\ref{l2.2.6}) into (\ref{l2.2.3}), and first letting $k\rightarrow\infty$ and then letting $\sigma\rightarrow 0$, we immediately obtain (\ref{l2.2.0}).
\end{proof}

Our next result concerns the density of $C^{\infty}(\overline D;\mb R^2)$ in the space $V$ (defined in \eqref{V}) with respect to a norm that is similar to that on the space of BV functions. More precisely, we have

\begin{proposition}\label{prop2.2}
	Assume $v\in V$. There exists a sequence $\{v_k\} \subset V\cap C^{\infty}(\overline D;\mb R^2)$ such that
	\begin{equation}\label{prop2.2.1}
	v_k\rightarrow v \quad\text{ in } L^2(D;\mathbb R^2),
	\end{equation}
	and 
	\begin{equation}\label{prop2.2.2}
	|\text{curl}v_k|(D)\rightarrow |\text{curl}v|(D)
	\end{equation}
	as $k\rightarrow \infty$.
\end{proposition}

The above proposition is needed in Section 4. Such results might be well-known to experts. However, the author did not find a proof in the literature. For the sake of completion, we include a proof in the appendix. Essentially, the space $V$ carries a structure that is analogous to that on the space BV. We adapt standard approximation and extension techniques for BV functions to prove Proposition \ref{prop2.2}.

\section{Compactness and lower bound}

In this section, we prove the compactness and lower bound estimates in Theorem \ref{T1}. By multiplying out the term $\hat{\nabla}_{\hat A_{n}}u_n=\hat{\nabla} u_n-\imath \hat A_n u_n$, we write
\begin{equation}\label{splitting}
\begin{split}
\mathcal{G}_{LD}^{\epsilon, s}&(\{u_n\}_{n=0}^N, \vec{A})\\
&=s\sum^N_{n=0}E_{\epsilon}(u_n)-s\sum^N_{n=0} \int_\Omega(\hat{\nabla}u_n,\imath u_n)\cdot\hat{A}_nd\hat x+\frac{s}{2}\sum^N_{n=0} \int_\Omega|u_n|^2|\hat{A}_n|^2d\hat x\\
&+s\sum^{N-1}_{n=0}\int_\Omega\frac{1}{2\lambda^{2}s^2}\left| u_{n+1}-u_{n}e^{\imath \int_{ns}^{(n+1)s}A^{3}dx_{3}}\right|^{2}d\hat x\\
&+\frac{1}{2}\int_{\mathbb{R}^3}\left|\nabla\times{\vec{A}}-h_{ex}\vec{e}_3\right|^{2}dx,
\end{split}
\end{equation}
where we denote
\begin{equation}\label{E_epsilon}
E_{\epsilon}(u)=\frac{1}{2}\int_{\Omega}\left[ |\hat\nabla u|^2+\frac{(1-|u|^2)^2}{2\epsilon^2}\right] d\hat x,
\end{equation}
the simplified two-dimensional Ginzburg-Landau energy without magnetic field. A similar decomposition as in \eqref{splitting} was used for the Ginzburg-Landau energy. (See \cite{JS} and \cite{BJOS1}.) The decomposition \eqref{splitting} allows us to separate the energy of the magnetic terms from the two-dimensional Ginzburg-Landau energies on the layers. 

\subsection{The Jacobian estimate}

For the two-dimensional Ginzburg-Landau energies on the layers, a key step in the analysis is to separate the energy from the Jacobian (defined in \eqref{ja}) by showing that its energy is concentrated in small regions with total measure tending to zero. Such Jacobian estimates were proved for the Ginzburg-Landau energy in \cite{JS1} and \cite{SS2} (see also \cite{ABO}). The proof requires suitable upper bound on the two-dimensional Ginzburg-Landau energy. In Lemma \ref{lem2.2}, using the upper bound on the Lawrence-Doniach energy (\ref{thm1.5}) and Lemma \ref{l2.2}, we are able to show that the desired upper bound on the two-dimensional Ginzburg-Landau energy holds on most of the layers. Then following a standard slicing argument that was used in \cite{JS1} and \cite{SS2}, we are able to prove a Jacobian estimate for the Lawrence-Doniach energy (see \eqref{lem2.3.3}). The main result of this section is the following

\begin{theorem}\label{lem2.3}
	Assume $\lim_{\epsilon\rightarrow 0}\frac{h_{ex}}{|\ln\epsilon|}=h_0$ for some $0\leq h_0<\infty$. For any sequence $(\{u^{\epsilon}_n\},\vec{A}^{\epsilon,s})\in [H^1(\Omega;\mathbb{C})]^{N+1}\times K$ such that
	\begin{equation}\label{2.1}
	\mathcal{G}_{LD}^{\epsilon,s}(\{u_n^{\eps}\},\vec{A}^{\epsilon,s}) \leq C_0|\ln\epsilon|^2,
	\end{equation}
	where $C_0$ is a constant independent of $\epsilon$ and $s$, we have, up to a subsequence as $(\epsilon,s) \rightarrow (0,0)$,
	\begin{equation}\label{lem2.3.1}
	\sum\limits_{n=0}^{N-1}\frac{j(u_n^{\eps})}{|u_n^{\eps}||\ln\eps|}\chi_n \rightharpoonup v \text{ in } L^2(D;\mathbb{R}^2), \quad \frac{j^{\epsilon,s}}{|\ln\epsilon|}\rightharpoonup v \text{ in } L^{\frac{4}{3}}(D;\mathbb{R}^2),
	\end{equation}
	and
	\begin{equation}\label{lem2.3.2}
	\frac{J^{\epsilon,s}}{|\ln\epsilon|}\rightharpoonup w \text{ in } (C^{0,\alpha}_c(D))^{*} \text{ for all } 0<\alpha\leq 1
	\end{equation}
	for some $v \in L^2(D;\mathbb{R}^2)$ such that $w = \frac{1}{2} \mathrm{curl} v$ is a Radon measure. Moreover, there exist sets $Z_n^{\epsilon}\subset\Omega$ with $\lim s\sum_{n=0}^N|Z_n^{\epsilon}|=0$ such that
	\begin{equation}\label{lem2.3.3}
	\liminf_{(\epsilon,s)\rightarrow(0,0)}\frac{s}{2|\ln\epsilon|^2}\sum_{n=0}^N\int_{Z_n^{\epsilon}}|\hat{\nabla}u_n^{\epsilon}|^2d\hat x\geq|w|(D).
	\end{equation}
\end{theorem}

Note that Theorem \ref{lem2.3} does not rely on the assumption (\ref{h}). First we prove some auxiliary lemmas.

\begin{lemma}\label{l3.3}
	Under the assumptions of Theorem \ref{lem2.3}, for any bounded domain $\omega\subset\mathbb{R}^3$, we have
	\begin{equation*}
	\lVert \vec{A}^{\epsilon,s} \rVert_{H^1(\omega;\mathbb{R}^3)}^2\leq C|\ln\eps|^2
	\end{equation*}
	for some constant $C$ independent of $\epsilon$ and $s$.
\end{lemma}

\begin{proof}
	Using the assumption $\lim_{\epsilon\rightarrow 0}\frac{h_{ex}}{|\ln\epsilon|}=h_0<\infty$ and the fact that $\vec a$ is a fixed smooth vector field on $\mathbb{R}^3$, we have
	\begin{equation*}
	\begin{split}
	\int_{\omega}\lt|\nabla\vec{A}^{\epsilon,s}\rt|^2 dx &\leq 2\int_{\omega}\lt|\nabla\lt(\vec A^{\epsilon,s}-h_{ex}\vec a\rt)\rt|^2 dx + 2\int_{\omega}h_{ex}^2\lt|\nabla\vec a\rt|^2 dx \\
	&\leq 2\int_{\mathbb{R}^3}\lt|\nabla\lt(\vec A^{\epsilon,s}-h_{ex}\vec a\rt)\rt|^2 dx + C|\ln\eps|^2.
	\end{split}
	\end{equation*}
	Now using \eqref{H12} and \eqref{2.1} we have
	\begin{equation}\label{l3.3.1}
	\begin{split}
	2\int_{\mathbb{R}^3}\lt|\nabla\lt(\vec A^{\epsilon,s}-h_{ex}\vec a\rt)\rt|^2 dx&=2 \int_{\mathbb{R}^3}\lt|\nabla\times\lt(\vec A^{\epsilon,s}-h_{ex}\vec a\rt)\rt|^2 dx\\
	&\leq 4\mathcal{G}_{LD}^{\epsilon,s}(\{u_n^{\eps}\},\vec{A}^{\epsilon,s})\leq 4C_0|\ln\eps|^2.
	\end{split}
	\end{equation}
	Therefore we have
	\begin{equation}\label{l3.2.1}
	\int_{\omega}\lt|\nabla\vec{A}^{\epsilon,s}\rt|^2 dx \leq C|\ln\epsilon|^2.
	\end{equation}
	
	Using H\"{o}lder's inequality, \eqref{H11} and (\ref{l3.3.1}) we deduce that
	\begin{equation*}
	\lVert \vec A^{\epsilon,s}-h_{ex}\vec a\rVert_{L^2(\omega;\mathbb{R}^3)}^2 \leq C \lVert \vec A^{\epsilon,s}-h_{ex}\vec a\rVert_{L^6(\omega;\mathbb{R}^3)}^2 \leq 4C\lVert \vec A^{\epsilon,s}-h_{ex}\vec a\rVert_{\check H^1(\mb R^3;\mb R^3)}^2 \leq \tilde{C}|\ln\eps|^2.
	\end{equation*}
	Therefore we have
	\begin{equation}\label{l3.3.2}
	\int_{\omega}\lt|\vec A^{\epsilon,s}\rt|^2 dx \leq 2\int_{\omega}\lt|\vec A^{\epsilon,s}-h_{ex}\vec a\rt|^2 dx + 2\int_{\omega}\lt|h_{ex}\vec a\rt|^2 dx \leq C|\ln\eps|^2.
	\end{equation}
	Putting (\ref{l3.2.1}) and (\ref{l3.3.2}) together concludes the proof of the lemma.
\end{proof}

\begin{lemma}\label{l3.2}
	Under the assumptions of Theorem \ref{lem2.3}, we have
	\begin{equation}\label{2.2.5}
	s\sum^{N-1}_{n=0}\int_\Omega |\hat{A}^{\epsilon,s}_n|^2 d\hat x \leq C|\ln\eps|^2
	\end{equation}
	and
	\begin{equation}\label{2.2.6}
	s\sum^{N-1}_{n=0}\int_\Omega |\hat{A}^{\epsilon,s}_n|^4 d\hat x \leq C|\ln\eps|^4
	\end{equation}
	for constants $C$ independent of $\epsilon$ and $s$.
\end{lemma}

\begin{proof}
	First note that
	\begin{equation}\label{226}
	s\sum^{N-1}_{n=0}\int_\Omega |\hat{A}^{\epsilon,s}_n|^2 d\hat x \leq 2\sum_{n=0}^{N-1}\int_{ns}^{(n+1)s}\int_{\Omega}\left(|\hat A^{\epsilon,s}-\hat A^{\epsilon,s}_n|^2+|\hat A^{\epsilon,s}|^2\right) d\hat x dx_3.
	\end{equation}
	It follows from Lemmas \ref{l2.2} and \ref{l3.3} that
	\begin{equation}\label{2.2.7}
	\sum_{n=0}^{N-1}\int_{ns}^{(n+1)s}\int_{\Omega} |\hat A^{\epsilon,s}-\hat A^{\epsilon,s}_n|^2 d\hat x dx_3 \leq s^2\int_{D}\lt|\nabla\vec{A}^{\epsilon,s}\rt|^2dx\leq Cs^2 |\ln\eps|^2.
	\end{equation}
	Hence, \eqref{2.2.5} follows from \eqref{226}, (\ref{2.2.7}) and (\ref{l3.3.2}).
	
	To show (\ref{2.2.6}), one can use the Sobolev embedding theorem and the trace theorem in a similar way as in the proof of Lemma \ref{l2.2}, in particular, as in (\ref{l2.2.2}) and (\ref{l2.2.6}), to conclude that
	\begin{equation*}
	\lVert \hat{A}^{\epsilon,s}_n \rVert_{L^4(\Omega_n)}\leq C\lVert \hat{A}^{\epsilon,s}_n \rVert_{H^{\frac{1}{2}}(\Omega_n)} \leq C\lVert \hat{A}^{\epsilon,s} \rVert_{H^{1}(\omega)}
	\end{equation*}
	for some sufficiently large bounded domain $\omega\subset \mathbb{R}^3$ and some constant $C$ independent of $\eps$, $s$ and $n$. Hence, noting that $sN=L$ is fixed, we have
	\begin{equation}\label{l3.2.3}
	s\sum^{N-1}_{n=0}\int_\Omega |\hat{A}^{\epsilon,s}_n|^4 d\hat x\leq sN C\lVert \hat{A}^{\epsilon,s} \rVert_{H^{1}(\omega)}^4 = LC\lVert \hat{A}^{\epsilon,s} \rVert_{H^{1}(\omega)}^4.
	\end{equation}
	Finally, (\ref{2.2.6}) follows from (\ref{l3.2.3}) and Lemma \ref{l3.3}.
\end{proof}

\begin{lemma}\label{lem3.4}
	Under the assumptions of Theorem \ref{lem2.3}, we have
	\begin{equation}\label{lem3.4.1}
	s\sum^N_{n=0}\int_\Omega \lt(|u_n^{\eps}|^2-1\rt)|\hat{A}^{\epsilon,s}_n|^2 d\hat x\rightarrow 0
	\end{equation}
	and
	\begin{equation}\label{lem3.4.2}
	s\sum^N_{n=0}\int_\Omega \lt(|u_n^{\eps}|-1\rt)^2|\hat{A}^{\epsilon,s}_n|^2 d\hat x\rightarrow 0
	\end{equation}
	as $(\epsilon,s)\rightarrow(0,0)$.
\end{lemma}

\begin{proof}
	Using H\"{o}lder's and the Cauchy-Schwarz inequalities, we have
	\begin{equation}\label{lem2.2.1}
	\begin{split}
	&s\sum^N_{n=0}\int_\Omega \lt(|u_n^{\eps}|^2-1\rt)|\hat{A}^{\epsilon,s}_n|^2 d\hat x\\
	&\quad\quad\leq s\sum^N_{n=0} \lt(\int_{\Omega}\frac{\lt(|u_n^{\eps}|^2-1\rt)^2}{4\epsilon^2}d\hat x\rt)^{\frac{1}{2}}\lt(\int_{\Omega}4\epsilon^2|\hat{A}^{\epsilon,s}_n|^4d\hat x\rt)^{\frac{1}{2}}\\
	&\quad\quad\leq\lt(s\sum^N_{n=0}\int_{\Omega}\frac{\lt(|u_n^{\eps}|^2-1\rt)^2}{4\epsilon^2}d\hat x\rt)^{\frac{1}{2}}\lt(s\sum^N_{n=0}\int_{\Omega}4\epsilon^2|\hat{A}^{\epsilon,s}_n|^4d\hat x\rt)^{\frac{1}{2}}.
	\end{split}
	\end{equation}
	Using the energy bound (\ref{2.1}), we have
	\begin{equation}
	s\sum^N_{n=0}\int_{\Omega}\frac{\lt(|u_n^{\eps}|^2-1\rt)^2}{4\epsilon^2}d\hat x \leq C_0|\ln\eps|^2.
	\end{equation}
	Using (\ref{2.2.6}) we have
	\begin{equation}\label{lem2.2.2}
	4\epsilon^2s\sum^N_{n=0}\int_{\Omega}|\hat{A}^{\epsilon,s}_n|^4d\hat x \leq C\epsilon^2|\ln\eps|^4 \rightarrow 0 \quad\text{ as } (\epsilon,s)\rightarrow(0,0).
	\end{equation}
	Putting (\ref{lem2.2.1})-(\ref{lem2.2.2}) together we obtain (\ref{lem3.4.1}). The estimate (\ref{lem3.4.2}) follows from (\ref{lem3.4.1}) by noting that $\lt(|u_n^{\eps}|-1\rt)^4\leq \lt(|u_n^{\eps}|-1\rt)^2\lt(|u_n^{\eps}|+1\rt)^2=\lt(|u_n^{\eps}|^2-1\rt)^2$.
\end{proof}

\begin{lemma}\label{lem2.2}
	Under the assumptions of Theorem \ref{lem2.3}, we have
	\begin{equation*}
	s\sum^N_{n=0}E_{\epsilon}(u_n^{\epsilon})\leq C|\ln\epsilon|^2
	\end{equation*}
	for some constant $C$ independent of $\epsilon$ and $s$.
\end{lemma}

\begin{proof}
	Using the decomposition \eqref{splitting} and the upper bound \eqref{2.1}, we have
	\begin{equation}\label{2.2.1}
	s\sum^N_{n=0}E_{\epsilon}(u_n^{\epsilon})\leq C_0|\ln\epsilon|^2+s\sum^N_{n=0} \int_\Omega(\hat{\nabla}u_n^{\epsilon},\imath u_n^{\epsilon})\cdot\hat{A}^{\epsilon,s}_nd\hat x.
	\end{equation}
	Therefore we only need to estimate the term $s\sum^N_{n=0} \int_\Omega(\hat{\nabla}u_n^{\epsilon},\imath u_n^{\epsilon})\cdot\hat{A}^{\epsilon,s}_nd\hat x$. We follow an idea used in \cite{JS} and \cite{BJOS1}. First, by an elementary inequality, we have
	\begin{equation*}
	s\sum^N_{n=0} \int_\Omega(\hat{\nabla}u_n^{\epsilon},\imath u_n^{\epsilon})\cdot\hat{A}^{\epsilon,s}_nd\hat x\leq s\sum^N_{n=0}\int_\Omega\left(\frac{1}{4}|\hat{\nabla}u_n^{\epsilon}|^2+|u_n^{\eps}|^2|\hat{A}^{\epsilon,s}_n|^2\right)d\hat x.
	\end{equation*}
	Plugging the above into \eqref{2.2.1} and using the expression of $E_{\epsilon}$ in \eqref{E_epsilon}, we obtain
	\begin{equation*}
	\frac{s}{2}\sum^N_{n=0}E_{\epsilon}(u_n^{\epsilon})\leq C_0|\ln\epsilon|^2+s\sum^N_{n=0}\int_\Omega |u_n^{\eps}|^2|\hat{A}^{\epsilon,s}_n|^2 d\hat x.
	\end{equation*}
	It follows from (\ref{lem3.4.1}) and (\ref{2.2.5}) that
	\begin{equation*}
	s\sum^N_{n=0}\int_\Omega |u_n^{\eps}|^2|\hat{A}^{\epsilon,s}_n|^2 d\hat x = s\sum^N_{n=0}\int_\Omega \lt(|u_n^{\eps}|^2-1\rt)|\hat{A}^{\epsilon,s}_n|^2 d\hat x + s\sum^N_{n=0}\int_\Omega |\hat{A}^{\epsilon,s}_n|^2 d\hat x\leq C|\ln\eps|^2.
	\end{equation*}
	This concludes the proof of the lemma.
\end{proof}

The above Lemma \ref{lem2.2} enables us to prove a Jacobian estimate for the Lawrence-Doniach energy. The proof is very similar to that for the higher dimensional Ginzburg-Landau energy in \cite{JS1} and \cite{SS2}. Here we provide the proof for the convenience of the reader. We first recall a covering result from \cite{SS2} which is convenient for our purposes. As in \cite{SS2}, we choose a sequence $M(\epsilon)$ such that, for all $\alpha > 0$,
\begin{equation}\label{M}
\lim\limits_{\epsilon \rightarrow 0} \epsilon^{\alpha}M(\epsilon) = 0, \quad \lim\limits_{\epsilon\rightarrow 0}\frac{|\ln\epsilon|}{M(\epsilon)^{\alpha}} = 0,
\end{equation}
and 
$$\ln M(\epsilon) = o(|\ln\epsilon|) \text{ as } \epsilon \rightarrow 0.$$
Then we have the following proposition which is Proposition 4.2 in \cite{SS2}:

\begin{proposition}[Sandier-Serfaty \cite{SS2}]\label{prop3.3}
	If $u$ satisfies $E_{\epsilon}(u) < c_0M(\epsilon)$ for some constant $c_0$ and some $0<\epsilon<1$, then there exist disjoint balls $B^1, \dddot\ , B^{l}$ with $B^i = B(a^i, r^i)$ such that, denoting $\tilde \Omega = \{x\in\Omega: \text{dist}(x,\partial\Omega) > \epsilon \}$, we have
	\begin{enumerate}
		\item \begin{equation}\label{lem2.3.18}
		\sum_{i} r^i \leq \frac{1}{M(\epsilon)}.
		\end{equation}
		\item For any $x\in \tilde{\Omega}\setminus \cup_i B^i$, $\lt||u(x)| - 1\rt| \leq \frac{2}{M(\epsilon)}$.
		\item If $B^i\subset\tilde{\Omega}$, 
		\begin{equation}\label{lem2.3.12}
		\frac{1}{2}\int_{B^i}|\hat\nabla u|^2 \geq \pi|d^i||\ln\epsilon|(1-o_{\epsilon}(1)),
		\end{equation}
		where $d^i=\deg(u,\partial B^i)$, and $o_{\eps}(1)\rightarrow 0$ as $\eps\rightarrow 0$.
		\item Letting $\mu^{\epsilon}=\pi\sum_{\{i:a^i\in\tilde{\Omega}\}}d^i\delta_{a^i}$, we have 
		\begin{equation}\label{lem2.3.6}
		\parallel Ju-\mu^{\epsilon} \parallel_{(C^{0,1}_c)^*}\leq C\frac{E_{\epsilon}(u)}{M(\epsilon)}
		\end{equation}
		for some constant $C$ depending only on $c_0$.
	\end{enumerate}
\end{proposition}

Using the above proposition, we are able to prove Theorem \ref{lem2.3}, whose proof follows closely that of Proposition 4.3 in \cite{SS2}.

\begin{proof}[Proof of Theorem \ref{lem2.3}]
	First note that 
	\begin{equation}\label{lem2.3.4}
	|j(u_n^{\epsilon})| = |(\imath u_n^{\epsilon}, \hat\nabla u_n^{\epsilon})| \leq |u_n^{\eps}|\cdot|\hat\nabla u_n^{\epsilon}|
	\end{equation}
	for all $n=0,1,\dddot\ ,N$. We deduce from Lemma \ref{lem2.2} that $\lt \{\sum\frac{j(u_n^{\eps})}{|u_n^{\eps}||\ln\eps|}\chi_n \rt \}$ forms a bounded sequence in $L^2(D;\mathbb{R}^2)$. Therefore, there exists some $v\in L^2(D;\mathbb{R}^2)$ such that
	\begin{equation}\label{lem2.3.01}
	\sum\limits_{n=0}^{N-1}\frac{j(u_n^{\eps})}{|u_n^{\eps}||\ln\eps|}\chi_n \rightharpoonup v \text{ in } L^2(D;\mathbb{R}^2).
	\end{equation}
	We write
	\begin{equation}\label{lem2.3.02}
	\sum\limits_{n=0}^{N-1}\frac{j(u_n^{\eps})}{|\ln\eps|}\chi_n=\sum\limits_{n=0}^{N-1}\frac{j(u_n^{\eps})}{|u_n^{\eps}||\ln\eps|}\lt(|u_n^{\eps}|-1\rt)\chi_n + \sum\limits_{n=0}^{N-1}\frac{j(u_n^{\eps})}{|u_n^{\eps}||\ln\eps|}\chi_n.
	\end{equation}
	It is clear that $\lt(|u_n^{\eps}|-1\rt)^4\leq\lt(|u_n^{\eps}|-1\rt)^2\lt(|u_n^{\eps}|+1\rt)^2=\lt(|u_n^{\eps}|^2-1\rt)^2$. Therefore, using Lemma \ref{lem2.2}, we obtain
	\begin{equation}\label{lem2.3.03}
	s\sum_{n=0}^{N-1}\int_{\Omega}\lt(|u_n^{\eps}|-1\rt)^4d\hat x \leq s\sum_{n=0}^{N-1}\int_{\Omega}\lt(|u_n^{\eps}|^2-1\rt)^2d\hat x \leq C\epsilon^2|\ln\eps|^2\rightarrow 0.
	\end{equation}
	We conclude from (\ref{lem2.3.01}) and (\ref{lem2.3.03}) that
	\begin{equation*}
	\sum\limits_{n=0}^{N-1}\frac{j(u_n^{\eps})}{|u_n^{\eps}||\ln\eps|}\lt(|u_n^{\eps}|-1\rt)\chi_n \rightharpoonup 0 \text{  in } L^{\frac{4}{3}}(D;\mathbb{R}^2).
	\end{equation*}
	This together with (\ref{lem2.3.01})-(\ref{lem2.3.02}) implies that
	\begin{equation*}
	\frac{j^{\epsilon,s}}{|\ln\epsilon|}\rightharpoonup v \text{ in } L^{\frac{4}{3}}(D;\mathbb{R}^2).
	\end{equation*}
	
	To show \eqref{lem2.3.2}, we use an interpolation argument first used in \cite{JS1}. First we show compactness of the sequence $\bigl \{ \frac{J^{\epsilon,s}}{|\ln\epsilon|} \bigr \}$ in $(C^{0,1}_c(D))^{*}$ by a slicing argument. We define $\mu^{\epsilon,s}=\sum_{n=0}^{N-1} \mu_n^{\epsilon}\chi_n(x_3)$ with
	\begin{equation}\label{lem2.3.15}
	\mu_n^{\epsilon}=
	\begin{cases}
	\pi\sum_{\{i:a_n^i\in\tilde{\Omega}_n\}}d_n^i\delta_{a_n^i} & \text{if } E_{\epsilon}(u_n^{\epsilon}) < c_0M(\epsilon),\\
	0 & \text{otherwise},
	\end{cases}
	\end{equation}
	where $a_n^i$ and $d_n^i$ are as in Proposition \ref{prop3.3}, and the functions $\chi_n$ are defined in \eqref{chi}. Let $\nu_n^{\epsilon}=Ju_n^{\epsilon}-\mu_n^{\epsilon}$ and $\nu^{\epsilon,s}=J^{\epsilon,s}-\mu^{\epsilon,s}$. Let $\varphi\in C_c^{0,1}(D)$ be a test function. We have
	\begin{equation}\label{lem2.3.5}
	\int_{D}\varphi d\nu^{\epsilon,s}=\sum_{n=0}^{N-1}\int_{ns}^{(n+1)s}\int_{\Omega}\varphi(\hat x,x_3)d\nu_n^{\epsilon} dx_3.
	\end{equation}
	If $E_{\epsilon}(u_n^{\epsilon}) < c_0M(\epsilon)$, then we obtain from \eqref{lem2.3.6} that 
	\begin{equation}\label{lem2.3.7}
	\int_{\Omega}\varphi(\hat x,x_3)d\nu_n^{\epsilon} \leq C \parallel \varphi \parallel_{C^{0,1}(D)}\frac{E_{\epsilon}(u_n^{\epsilon})}{M(\epsilon)}
	\end{equation}
	for all $x_3 \in (ns,(n+1)s)$. If $E_{\epsilon}(u_n^{\epsilon}) \geq c_0M(\epsilon)$, then, by definition, we have $\mu_n^{\epsilon}=0$, and therefore, $\nu_n^{\epsilon}=Ju_n^{\epsilon}$. Using an integration by parts, we have
	\begin{equation}\label{lem2.3.04}
	\int_{\Omega}\varphi d\nu_n^{\epsilon}=-\frac{1}{2}\int_{\Omega}\hat{\nabla}^{\perp}\varphi \cdot j(u_n^{\epsilon}) d\hat x,
	\end{equation}
	where recall that $\hat{\nabla}^{\perp} = (-\partial_2,\partial_1)$. Using \eqref{lem2.3.4} we have
	\begin{equation*}
	\int_{\Omega}\lt|j(u_n^{\epsilon}) \rt|d\hat x \leq \int_{\Omega}|u_n^{\eps}|\cdot|\hat\nabla u_n^{\epsilon}|d\hat x = \int_{\Omega}\lt(|u_n^{\eps}|-1\rt)|\hat\nabla u_n^{\epsilon}|d\hat x + \int_{\Omega}|\hat\nabla u_n^{\epsilon}|d\hat x.
	\end{equation*}
	Using H\"{o}lder's inequality, we have
	\begin{equation*}
	\begin{split}
	\int_{\Omega}\lt(|u_n^{\eps}|-1\rt)&|\hat\nabla u_n^{\epsilon}|d\hat x \leq 2\eps\lt(\int_{\Omega}\frac{\lt(|u_n^{\eps}|-1\rt)^2}{4\epsilon^2}d\hat x\rt)^{\frac{1}{2}}\lt(\int_{\Omega}|\hat\nabla u_n^{\epsilon}|^2d\hat x\rt)^{\frac{1}{2}}\\
	&\leq 2\eps\lt(\int_{\Omega}\frac{\lt(|u_n^{\eps}|^2-1\rt)^2}{4\epsilon^2}d\hat x\rt)^{\frac{1}{2}}\lt(\int_{\Omega}|\hat\nabla u_n^{\epsilon}|^2d\hat x\rt)^{\frac{1}{2}}\leq C\epsilon\cdot E_{\epsilon}(u_n^{\epsilon})
	\end{split}
	\end{equation*}
	and
	\begin{equation*}
	\int_{\Omega}|\hat\nabla u_n^{\epsilon}|d\hat x \leq C E_{\epsilon}(u_n^{\epsilon})^{\frac{1}{2}}.
	\end{equation*}
	Therefore we have
	\begin{equation*}
	\int_{\Omega}\lt|j(u_n^{\epsilon}) \rt|d\hat x \leq C\lt(\epsilon E_{\epsilon}(u_n^{\epsilon})+E_{\epsilon}(u_n^{\epsilon})^{\frac{1}{2}}\rt).
	\end{equation*}
	It follows from this and (\ref{lem2.3.04}) that
	\begin{equation*}
	\begin{split}
	\lt|\int_{\Omega}\varphi d\nu_n^{\epsilon}\rt| \leq 
	\frac{\parallel \varphi \parallel_{C^{0,1}(D)}}{2} \int_{\Omega} \lt|j(u_n^{\epsilon}) \rt| d\hat x \leq C\parallel \varphi \parallel_{C^{0,1}(D)} \lt(\epsilon E_{\epsilon}(u_n^{\epsilon})+E_{\epsilon}(u_n^{\epsilon})^{\frac{1}{2}}\rt).
	\end{split}
	\end{equation*}
	The assumption $E_{\epsilon}(u_n^{\epsilon}) \geq c_0M(\epsilon)$ yields $E_{\epsilon}(u_n^{\epsilon})^{\frac{1}{2}} \leq c_0^{-\frac{1}{2}}\frac{E_{\epsilon}(u_n^{\epsilon})}{M(\epsilon)^{\frac{1}{2}}}$. Also from (\ref{M}) it is clear that $\epsilon<\frac{1}{M(\epsilon)^{\frac{1}{2}}}$ for $\eps$ sufficiently small. Therefore, we deduce that
	\begin{equation}\label{lem2.3.9}
	\lt|\int_{\Omega}\varphi d\nu_n^{\epsilon}\rt| \leq C\parallel \varphi \parallel_{C^{0,1}(D)} \frac{E_{\epsilon}(u_n^{\epsilon})}{M(\epsilon)^{\frac{1}{2}}},
	\end{equation}
	provided $E_{\epsilon}(u_n^{\epsilon}) \geq c_0M(\epsilon)$. It follows from \eqref{lem2.3.5}, (\ref{lem2.3.7}), \eqref{lem2.3.9}, and the choice of $M(\epsilon)$ that
	\begin{equation*}
	\lt|\int_{D}\varphi d\nu^{\epsilon,s}\rt| \leq C \parallel \varphi \parallel_{C^{0,1}(D)} s\sum_{n=0}^{N-1} \frac{E_{\epsilon}(u_n^{\epsilon})}{M(\epsilon)^{\frac{1}{2}}}
	\end{equation*}
	for some constant $C$ independent of $\epsilon$ and $s$. The above together with Lemma \ref{lem2.2} implies that
	\begin{equation}\label{lem2.3.10}
	\parallel \nu^{\epsilon,s} \parallel_{(C_c^{0,1}(D))^*} \leq Cs\sum_{n=0}^{N-1} \frac{E_{\epsilon}(u_n^{\epsilon})}{M(\epsilon)^{\frac{1}{2}}}\leq \frac{C|\ln\eps|^2}{M(\epsilon)^{\frac{1}{2}}}
	\end{equation}
	for some constant $C$ as above. \blue{Using \eqref{lem2.3.15}, \eqref{lem2.3.12} and Lemma \ref{lem2.2}, we have} 
	\begin{equation}\label{lem2.3.100}
	\blue{\lt\lVert \mu^{\epsilon,s}\rt\rVert_{(C_c^{0}(D))^*} \leq s\sum_{n}\sum_{i}\pi|d_n^i| \leq C\frac{s}{|\ln\eps|}\sum_n \int_{\Omega}|\hat{\nabla}u_n^{\epsilon}|^2 d\hat x\leq C|\ln\eps|.}
	\end{equation}
	Using the definition of the Jacobian, it is clear that $|Ju_n^{\epsilon}| \leq C|\hat{\nabla}u_n^{\epsilon}|^2$. So we obtain from Lemma \ref{lem2.2} that
	\begin{equation*}
	\lVert J^{\eps,s}\rVert_{(C_c^{0}(D))^*} \leq C|\ln\eps|^2.
	\end{equation*}
	It follows that
	\begin{equation}
	\label{lem2.3.14}
	\parallel \nu^{\epsilon,s} \parallel_{(C_c^{0}(D))^*}\leq \lVert J^{\eps,s}\rVert_{(C_c^{0}(D))^*} + \lVert \mu^{\eps,s}\rVert_{(C_c^{0}(D))^*} \leq C|\ln\epsilon|^2.
	\end{equation}
	By Lemma 3.3 in \cite{JS1}, for any $0<\alpha<1$, we have
	\begin{equation}\label{lem2.3.11}
	\parallel \nu^{\epsilon,s} \parallel_{(C_c^{0,\alpha}(D))^*} \leq C\parallel \nu^{\epsilon,s} \parallel_{(C_c^{0}(D))^*}^{1-\alpha} \parallel \nu^{\epsilon,s} \parallel_{(C_c^{0,1}(D))^*}^{\alpha}.
	\end{equation}
	Therefore, we deduce from \eqref{M} and \eqref{lem2.3.10}, (\ref{lem2.3.14})-\eqref{lem2.3.11} that
	\begin{equation}\label{lem2.3.17}
	\lt\lVert \frac{\nu^{\epsilon,s}}{|\ln\epsilon|} \rt\rVert_{(C_c^{0,\alpha}(D))^*} \rightarrow 0 \quad\text{ as } (\epsilon,s) \rightarrow (0,0).
	\end{equation}
	\blue{It is clear from (\ref{lem2.3.100}) that $\{\frac{\mu^{\epsilon,s}}{|\ln\epsilon|}\}$ forms a bounded sequence in $(C_c^{0}(D))^*$.} The compact embedding $(C_c^{0}(D))^* \subset\subset (C_c^{0,\alpha}(D))^*$ (see Lemma 3.4 in \cite{JS1}) implies the precompactness of $\{\frac{\mu^{\epsilon,s}}{|\ln\epsilon|}\}$ in $(C_c^{0,\alpha}(D))^*$. Therefore, the precompactness of $\{\frac{J^{\epsilon,s}}{|\ln\epsilon|}\}$ in $(C_c^{0,\alpha}(D))^*$ follows from \eqref{lem2.3.17}.
	
	Finally, to show \eqref{lem2.3.3}, we let $Z_n^{\epsilon} = \cup_i B_n^i$, where $B_n^i$ are the balls as in Proposition \ref{prop3.3}. (If $E_{\epsilon}(u_n^{\epsilon}) \geq c_0M(\epsilon)$, we simply set $Z_n^{\epsilon} = \emptyset$.) By \eqref{lem2.3.18}, we have $|Z_n^{\epsilon}| \leq \frac{C}{M(\epsilon)^2}$, from which we immediately have 
	$$\lim_{(\epsilon,s) \rightarrow (0,0)} s\sum_{n} |Z_n^{\epsilon}| = 0.$$
	Let $\varphi\in C^1_c(D)$ be a test function. It follows from \eqref{lem2.3.2} and \eqref{lem2.3.17} that
	\begin{equation}\label{lem2.3.19}
	\int_{D}\varphi d w = \lim_{(\eps,s)\rightarrow(0,0)}\int_{D}\varphi d\frac{J^{\eps,s}}{|\ln\eps|} = \lim_{(\eps,s)\rightarrow(0,0)}\int_{D}\varphi d\frac{\mu^{\eps,s}}{|\ln\eps|}.
	\end{equation}
	By \eqref{lem2.3.12}, we have
	\begin{equation}\label{lem2.3.21}
	\begin{split}
	\int_{D}\varphi d\frac{\mu^{\eps,s}}{|\ln\eps|} &\leq \sup|\varphi|\cdot s\sum_{n}\sum_{i}\frac{\pi|d_n^i|}{|\ln\eps|}\\
	&\leq \sup|\varphi|\cdot \frac{s}{2|\ln\eps|^2}\sum_n \int_{Z_n^{\eps}}|\hat{\nabla}u_n^{\epsilon}|^2 d\hat x + o_{\eps}(1)\sup|\varphi|.
	\end{split}
	\end{equation}
	By putting \eqref{lem2.3.21} and \eqref{lem2.3.19} together, and taking the supremum over all $\varphi$ such that $\sup|\varphi|\leq 1$, we immediately obtain \eqref{lem2.3.3}.
\end{proof}

\subsection{Proof of lower bound in Theorem \ref{T1}}

Now we prove the lower bound estimate in Theorem \ref{T1}. The proof follows a standard idea used for the Ginzburg-Landau energy. Here the treatment of the magnetic potential is slightly different due to the layered structure of the problem.

\begin{proof}[Proof of \eqref{thm1.3}]
	First, by a short argument using \eqref{lem2.3.1}, \eqref{lem2.3.3} and (\ref{lem2.3.4}), it is easy to show
	\begin{equation}\label{2.4.1}
	\liminf_{(\epsilon,s)\rightarrow(0,0)} \frac{s\sum^N_{n=0}E_{\epsilon}(u_n^{\eps})}{|\ln\epsilon|^2} \geq \frac{1}{2}\parallel v \parallel_{L^2(D)}^2 + |w|(D).
	\end{equation}
	
	Let us focus on the terms involving the magnetic potential in the decomposition \eqref{splitting}. By \eqref{2.1}, we have
	$$\frac{1}{2}\int_{\mathbb R^3}\lt|\nabla\times(\vec A^{\epsilon,s} - h_{ex}\vec a)\rt|^2 dx \leq C|\ln\epsilon|^2.$$
	Using arguments similar to those in Proposition \ref{prop2.1}, we have that, up to a subsequence,
	\begin{equation}\label{2.4.5}
	\frac{\vec A^{\epsilon,s} - h_{ex}\vec a}{|\ln\epsilon|} \rightharpoonup \vec A-h_0\vec a \text{\quad in } \check H^1(\mathbb R^3;\mathbb R^3)
	\end{equation}
	for some $\vec A \in h_0\vec a + \check H^1(\mathbb R^3;\mathbb R^3)$, and that
	\begin{equation}\label{2.4.6}
	\liminf_{(\epsilon,s)\rightarrow(0,0)} \frac{1}{2|\ln\epsilon|^2}\int_{\mathbb R^3}\lt|\nabla\times(\vec A^{\epsilon,s} - h_{ex}\vec a)\rt|^2 dx \geq \frac{1}{2}\int_{\mathbb R^3}\lt|\nabla\times(\vec A - h_0 \vec a)\rt|^2 d x.
	\end{equation}
	
	Recall the definition of the functions $\chi_n$ in \eqref{chi}. Using Lemmas \ref{l2.2} and \ref{l3.3}, we have
	\begin{equation*}
	\lt\lVert \sum_n\hat A^{\eps,s}_n\chi_n-\hat A^{\eps,s}\rt\rVert_{L^2(D)}^2\leq s^2\int_{D}\lt|\nabla \vec A^{\eps,s}\rt|^2 dx \leq Cs^2|\ln\eps|^2.
	\end{equation*}
	It follows that
	\begin{equation*}
	\sum_n\frac{\hat A^{\eps,s}_n}{|\ln\eps|}\chi_n \rightarrow \frac{\hat A^{\epsilon,s}}{|\ln\eps|} \quad\text{ in } L^2(D;\mathbb{R}^2).
	\end{equation*}
	We deduce from \eqref{2.4.5} and the compact Sobolev embedding theorem that, up to a subsequence,
	\begin{equation*}
	\frac{\hat A^{\epsilon,s}}{|\ln\eps|}\rightarrow \hat A\quad\text{ in } L^2(D;\mathbb{R}^2).
	\end{equation*}
	Therefore, we obtain
	\begin{equation}\label{2.4.8}
	\sum_n\frac{\hat A^{\eps,s}_n}{|\ln\eps|}\chi_n \rightarrow \hat A \quad\text{ in } L^2(D)
	\end{equation}
	as $(\epsilon,s)\rightarrow(0,0)$. Now we write
	\begin{equation*}
	\begin{split}
	&\frac{s}{|\ln\epsilon|^2} \sum^N_{n=0} \int_\Omega j(u_n^{\eps})\cdot\hat{A}^{\eps,s}_nd\hat x \\
	&\quad\quad= s \sum^N_{n=0} \int_\Omega\frac{j(u_n^{\eps})}{|u_n^{\eps}||\ln\eps|}\cdot\lt(|u_n^{\eps}|-1\rt)\frac{\hat{A}^{\eps,s}_n}{|\ln\eps|}d\hat x + s \sum^N_{n=0} \int_\Omega\frac{j(u_n^{\eps})}{|u_n^{\eps}||\ln\eps|}\cdot \frac{\hat{A}^{\eps,s}_n}{|\ln\eps|}d\hat x.
	\end{split}
	\end{equation*}
	The above first term on the right side converges to zero because of (\ref{lem2.3.1}) and (\ref{lem3.4.2}). Using (\ref{lem2.3.1}) and (\ref{2.4.8}) to the second term implies
	\begin{equation}\label{2.4.9}
	\lim_{(\epsilon,s)\rightarrow(0,0)} \frac{s}{|\ln\epsilon|^2} \sum^N_{n=0} \int_\Omega(\hat{\nabla}u_n^{\eps},\imath u_n^{\eps})\cdot\hat{A}_n^{\epsilon,s}d\hat x = \int_{D} v \cdot \hat A d x.
	\end{equation}
	
	Next, we write
	\begin{equation*}
	\frac{s}{2}\sum^N_{n=0} \int_\Omega|u_n^{\eps}|^2|\hat{A}_n^{\epsilon,s}|^2d\hat x = \frac{s}{2}\sum^N_{n=0} \int_\Omega(|u_n^{\eps}|^2-1)|\hat{A}_n^{\epsilon,s}|^2d\hat x + \frac{s}{2}\sum^N_{n=0} \int_\Omega|\hat{A}_n^{\epsilon,s}|^2d\hat x.
	\end{equation*}
	It follows from (\ref{lem3.4.1}) and \eqref{2.4.8} that
	\begin{equation}\label{2.4.15}
	\lim_{(\epsilon,s)\rightarrow (0,0)}\frac{s}{2|\ln\epsilon|^2}\sum^N_{n=0} \int_\Omega|u_n^{\eps}|^2|\hat{A}_n^{\epsilon,s}|^2d\hat x = \frac{1}{2}\int_{D} |\hat A|^2 dx.
	\end{equation}
	
	Finally, we have
	\begin{equation}\label{2.4.25}
	\frac{s}{2\lambda^{2}s^2}\sum^{N-1}_{n=0}\int_{\Omega}\lt|u_{n}^{\eps}\rt|^{2}d\hat{x}=\frac{s}{2\lambda^{2}s^2}\sum^{N-1}_{n=0}\int_{\Omega}\lt(\lt|u_{n}^{\eps}\rt|^{2}-1\rt)d\hat{x}+\frac{sN|\Omega|}{2\lambda^{2}s^2},
	\end{equation}
	where $|\Omega|$ is the measure of the domain $\Omega$.
	Using H\"{o}lder's and the Cauchy-Schwarz inequalities, we have
	\begin{equation}\label{2.4.26}
	\begin{split}
	\frac{s}{2\lambda^{2}s^2}\sum^{N-1}_{n=0}\int_{\Omega}\lt(\lt|u_{n}^{\eps}\rt|^{2}-1\rt)d\hat{x} &\leq \frac{s|\Omega|^{\frac{1}{2}}}{2\lambda^{2}s^2}\sum^{N-1}_{n=0}\lt(\int_{\Omega}\lt(\lt|u_{n}^{\eps}\rt|^{2}-1\rt)^2d\hat{x}\rt)^{\frac{1}{2}}\\
	&\leq \frac{|\Omega|^{\frac{1}{2}}}{2\lambda^{2}s^2}\lt(s\sum^{N-1}_{n=0}\int_{\Omega}\lt(\lt|u_{n}^{\eps}\rt|^{2}-1\rt)^2d\hat{x}\rt)^{\frac{1}{2}} (sN)^{\frac{1}{2}}\\
	&\leq \frac{|\Omega|^{\frac{1}{2}}}{2\lambda^{2}s^2}\cdot C\eps|\ln\eps|\cdot L^{\frac{1}{2}}=C(|D|,\lambda)\frac{\epsilon|\ln\eps|}{s^2},
	\end{split}
	\end{equation}
	where in the above we have used Lemma \ref{lem2.2} and the facts that $sN=L$ and $L|\Omega|=|D|$. Now combining (\ref{2.4.25})-(\ref{2.4.26}) it is clear that
	\begin{equation*}
	\frac{1}{|\ln\eps|^2}\frac{s}{2\lambda^{2}s^2}\sum^{N-1}_{n=0}\int_{\Omega}\lt|u_{n}^{\eps}\rt|^{2}d\hat{x} \leq C(|D|,\lambda)\frac{\epsilon|\ln\eps|}{s^2|\ln\eps|^2} + \frac{|D|}{2\lambda^{2}}\frac{1}{s^2|\ln\eps|^2}\rightarrow 0,
	\end{equation*}
	since $\epsilon|\ln\eps|\rightarrow 0$ and $s^2|\ln\eps|^2\rightarrow \infty$ under the hypothesis (\ref{h}). It follows that
	\begin{equation}\label{2.4.23}
	\begin{split}
	&\frac{1}{|\ln\eps|^2}\frac{s}{2\lambda^{2}s^2}\sum^{N-1}_{n=0}\int_\Omega\lt|u_{n+1}^{\eps}-u_{n}^{\eps}e^{\imath \int_{ns}^{(n+1)s}(A^{\epsilon,s})^{3}dx_{3}}\rt|^{2}d\hat{x}\\
	&\quad\quad\quad\quad \leq \frac{1}{|\ln\eps|^2}\frac{2s}{2\lambda^{2}s^2}\sum^{N-1}_{n=0}\int_\Omega\lt(|u_{n+1}^{\eps}|^2+|u_{n}^{\eps}|^2\rt)d\hat{x}\rightarrow 0.
	\end{split}
	\end{equation}
	Putting \eqref{2.4.1}, \eqref{2.4.6}, \eqref{2.4.9}, \eqref{2.4.15} and \eqref{2.4.23} into \eqref{splitting}, we obtain \eqref{thm1.3}.
\end{proof}

\section{Construction of the order parameters}

In this and the next sections, we construct the recovery sequence satisfying the upper bound inequality in Theorem \ref{T1}. \blue{In the following sections, we omit the $(\tilde{\cdot})$ in the recovery sequence to avoid making the notations overly complicated.} This section is devoted to the construction of the order parameters $\{u_n\}$ on the layers. Essentially, on each layer, we follow constructions of test configurations for the two-dimensional Ginzburg-Landau energy used in \cite{SS1} and \cite{JS}. Here additional work is needed to take into account the limiting process as the interlayer distance $s$ tends to zero. This introduces several aspects of technical issues. First, this changes the locations of the layers and the vortices on the layers. Secondly, our results in this section do not rely on the assumption \eqref{h}. In other words, the estimates are independent of the relative size of $\eps$ and $s$. Lastly, we need to show compactness for discrete quantities that are in the form of sums of functions defined on the layers. As a result, some technical modifications of the two-dimensional constructions as well as compactness results are needed in the proof.

Given $v\in V\cap C^{\infty}(\overline D;\mb R^2)$ and $s=\frac{L}{N}$, for all $n=0,1,\dddot\ ,N-1$, we define $v_n(\hat x)=v(\hat x,ns)$ and $w_n(\hat x)=w(\hat x,ns)$, where $w(x)=\frac{1}{2}\cl v(x)$. Let
\begin{equation*}
v^s(x):=\sum_{n=0}^{N-1}v_n(\hat x)\chi_n(x_3), \quad \quad w^s(x):=\sum_{n=0}^{N-1}w_n(\hat x)\chi_n(x_3).
\end{equation*}
It is clear that
\begin{equation}\label{3.5}
v^s\rightarrow v \quad \text{ and } \quad w^s\rightarrow w \quad \text{uniformly as } s\rightarrow 0. 
\end{equation}
The main result of this section is the following

\begin{theorem}\label{p3.3}
	Given $v\in V\cap C^{\infty}(\overline D;\mb R^2)$, there exists a sequence $\{u_n^{\epsilon}\}_{n=0}^{N}\subset [H^1(\Omega;\mb C)]^{N+1}$ such that
	\begin{equation}\label{p3.3.1}
	v^{\epsilon,s}:=\frac{1}{|\ln\epsilon|}\sum_{n=0}^{N-1}j(u_n^{\epsilon})\chi_n\rightarrow v \quad\text{in } L^p(D;\mathbb R^2) \text{ for all } p<\frac{3}{2},
	\end{equation} 
	\begin{equation}\label{p3.3.2}
	w^{\epsilon,s}:=\frac{1}{|\ln\epsilon|}\sum_{n=0}^{N-1}J(u_n^{\epsilon})\chi_n\rightharpoonup w \quad\text{in } W^{-1,p}(D) \text{ for all } p<\frac{3}{2},
	\end{equation}
	and
	\begin{equation}\label{p3.3.3}
	\limsup_{(\epsilon,s)\rightarrow (0,0)}s\sum_{n=0}^{N}\frac{E_{\epsilon}(u_n^{\epsilon})}{|\ln\epsilon|^2}\leq \frac{\lVert v \rVert_{L^2(D)}^2}{2} + \lVert w \rVert_{L^1(D)}.
	\end{equation}
\end{theorem}

The above theorem generalizes Proposition 7.1 in \cite{JS}, in which a recovery sequence for the two-dimensional Ginzburg-Landau energy is constructed. Here we construct functions $u_n^{\epsilon}$ on the layers $\Omega_n = \Omega\times\{ns\}$ based on the two-dimensional constructions given in \cite{SS1} and \cite{JS}. We follow the approach given in \cite{JS}, which is more convenient in the three-dimensional context. As mentioned above, some technical modifications of the two-dimensional constructions are needed to obtain uniform estimates on the layers. Also, some compactness results need to be proved. (In particular, see Lemma \ref{l4.2}.)

Since $\Omega$ is a simply connected smooth domain, it is well-known that one can decompose $L^2(\Omega;\mb R^2)$ as a direct sum
\begin{equation*}
L^2(\Omega;\mb R^2) = \mathscr F\oplus \mathscr G,
\end{equation*}
where
\begin{equation*}
\mathscr F:= \{v: v=\hat{\nabla}\times f, f\in H^1(\Omega), f=0 \text{ on } \partial\Omega\},
\end{equation*}
and
\begin{equation*}
\mathscr G:= \{v: v=\hat{\nabla} g, g\in H^1(\Omega)\}.
\end{equation*}
(See, e.g., \cite{GR}.) For each $v_n$, we write
\begin{equation}\label{vn}
v_n = v_{n,1}+v_{n,2},
\end{equation}
where $v_{n,1}\in\mathscr F$ and $v_{n,2}\in\mathscr G$. Then we have
\begin{equation*}
v_{n,1} = (\partial_2 f, -\partial_1 f)
\end{equation*}
for some $f\in H^1(\Omega)$ with $f=0$ on $\partial\Omega$. By simple calculations, we have
\begin{equation*}
-\hat\Delta f = \cl v_{n,1} = \cl v_n = 2w_n,
\end{equation*}
where $\hat{\Delta}$ denotes the two-dimensional Laplacian. Therefore, we write $f=-2\hat\Delta^{-1}_D w_n$, and $v_{n,1}=-2\hat{\nabla}\times\hat\Delta^{-1}_D w_n$, where $\hat\Delta^{-1}_D$ is the inverse operator of the two-dimensional Laplacian with zero Dirichlet boundary data. Define
\begin{equation*}
v^s_1(x):=\sum_{n=0}^{N-1}v_{n,1}(\hat x)\chi_n(x_3) \quad\text{and}\quad v^s_2(x):=\sum_{n=0}^{N-1}v_{n,2}(\hat x)\chi_n(x_3).
\end{equation*}
First, we prove compactness of the sequences $\{v^s_1\}$ and $\{v^s_2\}$ in $L^2(D;\mb R^2)$.

\begin{lemma}\label{l4.2}
	Let $v^s_1$ and $v^s_2$ be as above. There exist $v_1, v_2\in L^2(D;\mb R^2)$ such that 
	\begin{equation*}
	v_1+v_2 = v,
	\end{equation*}
	\begin{equation*}
	v^s_1 \rightarrow v_1 \quad\text{ and }\quad v^s_2\rightarrow v_2 \quad \text{ in } L^2(D;\mb R^2) \text{ as } s\rightarrow 0,
	\end{equation*}
	and
	\begin{equation*}
	\langle v_1,v_2\rangle = \langle v_1^s,v_2 \rangle=\langle v_1,v_2^s\rangle = 0
	\end{equation*}
	for all $s$, where $\langle\cdot,\cdot\rangle$ denotes the inner product in $L^2(D;\mb R^2)$.
\end{lemma}

\begin{proof}
	Let $\{s_k\}$ be a sequence such that $\lim_{k\rightarrow \infty} s_k = 0$. We denote $v^k = v^{s_k}$ and $v^k_i=v^{s_k}_i$ for $i=1, 2$. It follows from \eqref{3.5} that
	\begin{equation}\label{l421}
	v^k_1 + v^k_2 = v^k \rightarrow v \quad\text{ in } L^2(D;\mb R^2) \text{ as } k\rightarrow\infty.
	\end{equation}
	Since $\lVert v^k \rVert_{2}^2 = \lVert v^k_1 \rVert_{2}^2+\lVert v^k_2 \rVert_{2}^2$, we deduce from \eqref{l421} that $\{v^k_1\}$ and $\{v^k_2\}$ form bounded sequences in $L^2$. Therefore, there exist $v_1, v_2\in L^2$ such that
	\begin{equation*}
	v^{k_j}_1\rightharpoonup v_1 \quad \text{and} \quad v^{k_j}_2\rightharpoonup v_2
	\end{equation*}
	for some subsequence $\{v^{k_j}_i\}$. It is clear that $v_1+v_2 = v$. From the weak convergence, we have
	\begin{equation}\label{l423}
	\langle v_1, v_2\rangle = \lim_{j\rightarrow\infty}\lim_{l\rightarrow\infty}\langle v^{k_j}_1,v^{k_l}_2\rangle = 0,
	\end{equation}
	since $v^{k_j}_1$ and $v^{k_l}_2$ are mutually orthogonal for all $k_j$ and $k_l$. Similarly, we have
	\begin{equation}\label{l424}
	\langle v^{k_j}_1,v_2\rangle = 0 \quad\text{and}\quad \langle v_1, v^{k_l}_2\rangle =0.
	\end{equation}
	It follows from \eqref{l421}-\eqref{l424} that
	\begin{equation*}
	\begin{split}
	0 = \lim_{j\rightarrow\infty}\langle v^{k_j}-v,v^{k_j}-v\rangle &= \lim_{j\rightarrow\infty}\langle (v^{k_j}_1-v_1)+(v^{k_j}_2-v_2),(v^{k_j}_1-v_1)+(v^{k_j}_2-v_2)\rangle\\
	&= \lim_{j\rightarrow\infty}\lt(\lVert v^{k_j}_1-v_1\rVert_2^2 + \lVert v^{k_j}_2-v_2\rVert_2^2\rt),
	\end{split}
	\end{equation*}
	from which we immediately deduce
	\begin{equation*}
	v^{k_j}_1\rightarrow v_1 \quad \text{and} \quad v^{k_j}_2\rightarrow v_2.
	\end{equation*}
	
	Suppose that another subsequence $\{v^{m_j}_i\}$ of $\{v^{k}_i\}$, for $i=1,2$, satisfies
	\begin{equation*}
	v^{m_j}_1 \rightarrow \tilde v_1 \quad \text{and} \quad v^{m_j}_2 \rightarrow \tilde v_2
	\end{equation*}
	for some $\tilde v_1$ and $\tilde v_2$. Clearly we have $\tilde v_1 + \tilde v_2 = v$. Using a similar argument as in \eqref{l423}, we have
	\begin{equation*}
	\langle \tilde v_1,\tilde v_2 \rangle = \langle \tilde v_1, v_2\rangle = \langle v_1, \tilde v_2 \rangle = 0.
	\end{equation*}
	It follows that
	\begin{equation*}
	\begin{split}
	0 &= \langle (v_1+v_2)-(\tilde v_1+\tilde v_2), (v_1+v_2)-(\tilde v_1+\tilde v_2)\rangle\\
	& = \langle (v_1-\tilde v_1)+(v_2-\tilde v_2), (v_1-\tilde v_1)+(v_2-\tilde v_2)\rangle\\
	&= \lVert v_1-\tilde v_1 \rVert_2^2 + \lVert v_2-\tilde v_2\rVert_2^2,
	\end{split}
	\end{equation*}
	from which we have $v_1 = \tilde v_1$ and $v_2 = \tilde v_2$. By a contradiction argument, this implies that the subsequence $\{v^{k_j}_i\}$ contains all but at most finitely many terms in $\{v^k_i\}$. Therefore, we have strong $L^2$ compactness for the whole sequence $\{v^k_i\}$ for $i=1,2$. 
\end{proof}

In the following, we construct recovery sequences for $v_1$ and $v_2$ respectively, as has been done in \cite{JS}.

\subsection{Recovery sequence for $v_1$}

\begin{lemma}\label{l4.3}
	Let $v_1$ be as in Lemma \ref{l4.2}. There exists a sequence $\{u_n^{\epsilon}\}_{n=0}^{N}\subset [H^1(\Omega;\mb C)]^{N+1}$ such that the conclusions of Theorem \ref{p3.3} hold with $v$ replaced by $v_1$.
\end{lemma}

To prove the above Lemma \ref{l4.3}, we will need a couple of auxiliary lemmas. We define $\xi(\hat x)=\frac{(x_2,-x_1)}{|\hat x|^2}$. Let $\eta:\mathbb R^2\rightarrow \mathbb R$ be a standard mollifier supported in the unit ball in $\mb R^2$ such that $\int_{\mathbb R^2}\eta=1$, and define $\eta^{\epsilon}:=\eta(\frac{x}{\epsilon})/\epsilon^2$. We define $q^{\epsilon}:\mathbb R^2\rightarrow \mathbb R$ by requiring
\begin{equation}\label{qe}
q^{\epsilon}(\hat x)\xi(\hat x) = \eta^{\epsilon}*\xi(\hat x).
\end{equation}
We have the following lemma taken from \cite{JS}, Lemma 7.3, which summarizes some useful properties of the function $q^{\epsilon}$.

\begin{lemma}[Jerrard-Soner \cite{JS}]\label{l3.4}
	The function $q^{\epsilon}$ is well-defined, smooth and radial, and has the following properties:
	\begin{equation}\label{l341}
	0\leq q^{\epsilon}\leq 1,\quad q^{\epsilon}(\hat x) = 1 \text{ whenever } |\hat x|\geq \epsilon,
	\end{equation}
	\begin{equation}\label{l342}
	q^{\epsilon}(\hat x) = q^1(\frac{\hat x}{\epsilon}),
	\end{equation}
	and
	\begin{equation*}
	\lVert q^{\epsilon}\xi \rVert_{\infty}\leq \frac{C}{\epsilon},\quad \int_{B_{\epsilon}}|q^{\epsilon}(\hat x)\xi(\hat x)|^2 d\hat x \leq C \text{ for some $C$ independent of $\epsilon$}.
	\end{equation*}
\end{lemma}

In the next lemma, we select vortices for each function $u_n^{\epsilon}$. 

\begin{lemma}\label{l3.5}
	Assume $w\in C^{\infty}(\overline D)$. For all $n=0,1,\dddot\ , N-1$, there exist families of points $\{a_{n,i}^{\epsilon}\}_{i=1}^{M_n^{\epsilon}}$ and integers $\sigma_{n,i}^{\epsilon}=\pm 1$ such that
	\begin{equation}\label{l351}
	w^{\epsilon,s}:=\sum_{n=0}^{N-1}w_n^{\epsilon}\chi_n\rightharpoonup w  \text{ weakly in } \mathcal{M} \text{ and strongly in } W^{-1,p} \text{ for all } p<\frac{3}{2}
	\end{equation}
	as $(\eps,s)\rightarrow (0,0)$, where $$w_n^{\epsilon}=\frac{\pi}{|\ln\epsilon|}\sum_{i=1}^{M_n^{\epsilon}}\sigma_{n,i}^{\epsilon}\delta_{a_{n,i}^{\epsilon}}.$$
	Moreover, the points $\{a_{n,i}^{\epsilon}\}$ can be chosen such that
	\begin{equation}\label{l353}
	|a_{n,i}^{\epsilon}-a_{n,j}^{\epsilon}|\geq c_{0}|\ln\epsilon|^{-\frac{1}{2}} \quad \forall i\ne j, \quad \text{dist}(a_{n,i}^{\epsilon},\partial\Omega)\geq c_{0}|\ln\epsilon|^{-\frac{1}{2}}\quad \forall i,
	\end{equation}
	where $c_0$ is some small constant depending on $\lVert w \rVert_{\infty}$. Finally, we also have the estimate
	\begin{equation}\label{l354}
	s\sum_{n=0}^{N-1}\lVert w_n^{\eps}-w_n \rVert_{W^{-1,p}(\Omega)}^p \rightarrow 0
	\end{equation}
	as $(\epsilon,s)\rightarrow (0,0)$.
\end{lemma}

\begin{proof}
	The proof relies on selecting points on each layer $\Omega_n=\Omega\times\{ns\}$. Let 
	\begin{equation*}
	\delta_{\eps}=|\ln\eps|^{-\frac{1}{4}}.
	\end{equation*}
	For each $\eps$, we define the family of squares in $\mb R^2$
	\begin{equation*}
	\{Q_i^{\eps}\}_i = \{[k\delta_{\eps},(k+1)\delta_{\eps}]\times[l\delta_{\eps},(l+1)\delta_{\eps}]: k,l\in\mathbb Z\}.
	\end{equation*}
	Denote
	\begin{equation*}
	I_{\eps} := \{i: Q_i^{\eps}\cap\Omega\ne\emptyset\}.
	\end{equation*}
	We write $\Omega_n=\cup_i \Omega_{n,i}^{\eps}$, where $\Omega_{n,i}^{\eps}=\blue{\Omega_n\cap (Q_i^{\eps}\times\{ns\})}$ for $i\in I_{\eps}$. In each $\Omega_{n,i}^{\eps}$, we define 
	\begin{equation}\label{l355}
	M_{n,i}^{\eps}=\begin{cases}
	\lfloor\frac{|\ln\eps|}{\pi}|\int_{\Omega_{n,i}^{\eps}}w_n d\hat x|\rfloor & \text{ if } \quad\text{dist}(\Omega_{n,i}^{\eps},\partial \Omega_n)>0,\\
	0 & \text{ otherwise},
	\end{cases}
	\end{equation}
	where $\lfloor x \rfloor$ is the integer part of a real number $x$, and 
	\begin{equation*}
	\sigma_{n,i}^{\eps}= \text{sgn}(\int_{\Omega_{n,i}^{\eps}}w_n d\hat x).
	\end{equation*}
	Note that $M_{n,i}^{\eps}\leq C\|w\|_{\infty}|\ln\eps|^{\frac{1}{2}}$ for some pure constant $C$. Therefore, in each $\Omega_{n,i}^{\eps}$, it is possible to place $M_{n,i}^{\eps}$ points $\{a_{n,ij}^{\eps}\}$ evenly such that \eqref{l353} is satisfied. Now we define
	\begin{equation*}
	w_n^{\epsilon}=\frac{\pi}{|\ln\epsilon|}\sum_i\sum_j\sigma_{n,i}^{\epsilon}\delta_{a_{n,ij}^{\epsilon}}.
	\end{equation*}
	Upon relabeling of $\{a_{n,ij}^{\eps}\}$, the above defined $w_n^{\epsilon}$ has the desired form.
	
	Next we verify \eqref{l351}. Note that, by the definition of $M_{n,i}^{\eps}$ in \eqref{l355}, we have
	\begin{equation*}
	M_n^{\eps}:=\sum_i M_{n,i}^{\eps} \leq \frac{|\ln\eps|}{\pi}\| w_n \|_{L^1(\Omega)}.
	\end{equation*}
	It follows from \eqref{3.5} that
	\begin{equation}\label{l357}
	|w^{\eps,s}|(D)\leq s\sum_{n=0}^{N-1}\frac{\pi}{|\ln\epsilon|}M_n^{\eps}\leq s\sum_{n=0}^{N-1} \| w_n \|_{L^1(\Omega)} \rightarrow \| w \|_{L^1(D)}.
	\end{equation}
	Therefore, it suffices to prove strong convergence of $w^{\eps,s}$ in $W^{-1,p}$ for all $p<\frac{3}{2}$. Fix some $p<\frac{3}{2}$ and take a test function $\phi\in W_0^{1,p'}(D)$ with $\|\phi\|_{W^{1,p'}(D)}\leq 1$, where $\frac{1}{p}+\frac{1}{p'}=1$. It is clear that $p'>3$. By the Sobolev embedding theorem, we have $\phi\in C^{0,\alpha}(D)$ for some $\alpha>0$ and $\|\phi\|_{C^{0,\alpha}(D)}\leq C$ for some constant $C$ depending on $p'$. (See, e.g., \cite{GT}.) It suffices to verify
	\begin{equation}\label{l3513}
	\sup_{\|\phi\|_{C^{0,\alpha}(D)}\leq 1} \left| \int_{D}\phi dw^{\eps,s} - \int_{D}\phi w^s dx \right| \rightarrow 0
	\end{equation}
	as $(\eps,s)\rightarrow (0,0)$. By the definitions of $w^{\eps,s}$ and $w^s$, we have
	\begin{equation}\label{l3512}
	\begin{split}
	\bigg| \int_{D}\phi dw^{\eps,s} &- \int_{D}\phi w^s dx \bigg|\\
	&= \left|\sum_{n=0}^{N-1}\int_{ns}^{(n+1)s}\left[ \frac{\pi}{|\ln\epsilon|}\sum_{i=1}^{M_n^{\epsilon}}\sigma_{n,i}^{\epsilon}\phi(a_{n,i}^{\epsilon}) - \int_{\Omega}\phi w_n d\hat x\right] dx_3 \right|\\
	&\leq\sum_{n=0}^{N-1}\int_{ns}^{(n+1)s}\sum_{j}\left| \frac{\pi}{|\ln\epsilon|}\sum_{i=1}^{M_{n,j}^{\epsilon}}\sigma_{n,i}^{\epsilon}\phi(a_{n,i}^{\epsilon}) - \int_{\Omega_{n,j}^{\eps}}\phi w_n d\hat x\right| dx_3 ,
	\end{split}
	\end{equation}
	where we use the convention that $\frac{\pi}{|\ln\epsilon|}\sum_{i=1}^{M_{n,j}^{\epsilon}}\sigma_{n,i}^{\epsilon}\phi(a_{n,i}^{\epsilon})=0$ if $M_{n,j}^{\eps}=0$. 
	
	If $\text{dist}(\Omega_{n,j}^{\eps},\partial \Omega_n)>0$ and $\text{sgn}(\int_{\Omega_{n,j}^{\eps}}w_n d\hat x)=1$, then we have
	\begin{equation*}
	\begin{split}
	&\left|\frac{\pi}{|\ln\epsilon|}\sum_{i=1}^{M_{n,j}^{\epsilon}}\sigma_{n,i}^{\epsilon}\phi(a_{n,i}^{\epsilon}) - \int_{\Omega_{n,j}^{\eps}}\phi w_n d\hat x\right| \\
	\leq&\sum_{i=1}^{M_{n,j}^{\epsilon}}|\phi(a_{n,i}^{\epsilon})|\left|\frac{\pi}{|\ln\epsilon|}-\frac{1}{M_{n,j}^{\epsilon}}\int_{\Omega_{n,j}^{\eps}} w_n d\hat x\right| + \frac{1}{M_{n,j}^{\epsilon}}\sum_{i=1}^{M_{n,j}^{\epsilon}}\int_{\Omega_{n,j}^{\eps}}\big|\phi(a_{n,i}^{\eps})-\phi\big| |w_n| d\hat x.
	\end{split}
	\end{equation*}
	Since $\|\phi\|_{C^{0,\alpha}(D)}\leq 1$, using the definition of $M_{n,j}^{\eps}$ in \eqref{l355}, we have
	\begin{equation*}
	\sum_{i=1}^{M_{n,j}^{\epsilon}}|\phi(a_{n,i}^{\epsilon})|\left|\frac{\pi}{|\ln\epsilon|}-\frac{1}{M_{n,j}^{\epsilon}}\int_{\Omega_{n,j}^{\eps}} w_n d\hat x\right|\leq M_{n,j}^{\epsilon}\left|\frac{\pi}{|\ln\epsilon|}-\frac{1}{M_{n,j}^{\epsilon}}\int_{\Omega_{n,j}^{\eps}} w_n d\hat x\right|\leq \frac{\pi}{|\ln\eps|},
	\end{equation*}
	and
	\begin{equation*}
	\frac{1}{M_{n,j}^{\epsilon}}\sum_{i=1}^{M_{n,j}^{\epsilon}}\int_{\Omega_{n,j}^{\eps}}\big|\phi(a_{n,i}^{\eps})-\phi\big| |w_n| d\hat x\leq \|w\|_{\infty}|\ln\eps|^{-\frac{\alpha}{4}-\frac{1}{2}}.
	\end{equation*}
	Hence, we have
	\begin{equation}\label{l358}
	\left|\frac{\pi}{|\ln\epsilon|}\sum_{i=1}^{M_{n,j}^{\epsilon}}\sigma_{n,i}^{\epsilon}\phi(a_{n,i}^{\epsilon}) - \int_{\Omega_{n,j}^{\eps}}\phi w_n d\hat x\right|\leq C|\ln\eps|^{-\frac{\alpha}{4}-\frac{1}{2}}
	\end{equation}
	for some constant $C$ depending on $\|w\|_{\infty}$. The cases when $\text{dist}(\Omega_{n,j}^{\eps},\partial \Omega_n)>0$ and $\text{sgn}(\int_{\Omega_{n,j}^{\eps}}w_n d\hat x)=-1$ or $0$ are similar. If $\text{dist}(\Omega_{n,j}^{\eps},\partial \Omega_n)=0$, then we have $w_n^{\eps}=0$ on $\Omega_{n,j}^{\eps}$, from which we deduce
	\begin{equation}\label{l3510}
	\left|\int_{\Omega_{n,j}^{\eps}}\phi d w_n^{\eps} - \int_{\Omega_{n,j}^{\eps}}\phi w_n d\hat x\right|=\left|\int_{\Omega_{n,j}^{\eps}}\phi w_n d\hat x\right|\leq \|w\|_{\infty}|\ln\eps|^{-\frac{1}{2}}.
	\end{equation}
	Since $\text{Card}\{j:\text{dist}(\Omega_{n,j}^{\eps},\partial \Omega_n)>0\}\leq C|\ln\eps|^{\frac{1}{2}}$ and $\text{Card}\{j:\text{dist}(\Omega_{n,j}^{\eps},\partial \Omega_n)=0\}\leq C|\ln\eps|^{\frac{1}{4}}$ for some constant $C$ depending on $\Omega$, we deduce from \eqref{l358}-\eqref{l3510} that
	\begin{equation}\label{l3511}
	\left|\int_{\Omega}\phi d w_n^{\eps} - \int_{\Omega}\phi w_n d\hat x\right|\leq C|\ln\eps|^{-\frac{\alpha}{4}}
	\end{equation}
	for some constant $C$ depending on $\Omega$ and $\|w\|_{\infty}$. By plugging \eqref{l3511} into \eqref{l3512}, and noting that $sN=L$ is fixed, we immediately obtain \eqref{l3513}, from which \eqref{l351} follows.
	
	Finally, to verify \eqref{l354}, note that, from \eqref{l3511}, one has
	\begin{equation*}
	\lVert w_n^{\eps}-w_n \rVert_{W^{-1,p}(\Omega)}\leq C|\ln\eps|^{-\frac{\alpha}{4}}.
	\end{equation*}
	Therefore, we have
	\begin{equation*}
	s\sum_{n=0}^{N-1}\lVert w_n^{\eps}-w_n \rVert_{W^{-1,p}(\Omega)}^p \leq sNC|\ln\eps|^{-\frac{p\alpha}{4}}=\tilde C|\ln\eps|^{-\frac{p\alpha}{4}}\rightarrow 0
	\end{equation*}
	as $\eps\rightarrow 0$, where $\tilde C=sNC$ depends only on $D$ and $\|w\|_{\infty}$.
\end{proof}

Now we provide the proof of Lemma \ref{l4.3}. We follow the ideas in the proof of Lemma 7.2 in \cite{JS}, with some technical modifications to take care of the limiting process as $s\rightarrow 0$. 

\begin{proof}[Proof of Lemma \ref{l4.3}]
	Fix a pair $(\epsilon,s)$. For each $n=0,1,\dddot\ ,N-1$, we have $w_n^{\epsilon}$ from Lemma \ref{l3.5}. Define $v_n^{\epsilon}=-2\hat{\nabla}\times\hat\Delta_{D}^{-1}w_n^{\epsilon}$. Next we define $\phi_n^{\epsilon}$ to be a multivalued function satisfying $\hat{\nabla}\phi_n^{\epsilon}=|\ln\epsilon|v_n^{\epsilon}$. It is easy to see that $\phi_n^{\epsilon}$ is well-defined modulo $2\pi$. Finally define
	\begin{equation*}
	u_n^{\epsilon}:=\rho_n^{\epsilon}e^{i\phi_n^{\epsilon}},\quad\quad \rho_n^{\epsilon}:=\prod q^{\epsilon}(\hat x-a_{n,i}^{\epsilon}),
	\end{equation*}
	where $q^{\epsilon}$ is defined in \eqref{qe} with properties summarized in Lemma \ref{l3.4}. In addition, we can set $u_N^{\eps}\equiv 1$. We show that the above defined functions $\{u_n^{\epsilon}\}$ satisfy the properties \eqref{p3.3.1}-\eqref{p3.3.3} with $v$ replaced by $v_1$.
	
	\textit{Step 1. Compactness:} For the above defined $u_n^{\epsilon}$, by direct calculations and the definition of $\phi^{\eps}_n$, we have
	\begin{equation*}
	j(u_n^{\epsilon})=(\imath u_n^{\eps},\hat{\nabla}u_n^{\eps})=|\rho_n^{\epsilon}|^2\hat{\nabla}\phi_n^{\epsilon}=|\ln\eps||\rho_n^{\eps}|^2v_n^{\eps},
	\end{equation*}
	from which we have
	\begin{equation}\label{p338}
	v^{\eps,s}=\frac{1}{|\ln\epsilon|}\sum_{n=0}^{N-1}j(u_n^{\epsilon})\chi_n=\sum_{n=0}^{N-1}|\rho_n^{\eps}|^2v_n^{\eps}\chi_n.
	\end{equation}
	Note that, by the definition of $\rho_n^{\eps}$ and the properties of $q^{\eps}$ in \eqref{l341}, we have
	\begin{equation*}
	\int_{\Omega}\lt(1-|\rho_n^{\eps}|^q\rt)^rd\hat x\leq \sum_{i=1}^{M_n^{\eps}}|B_{\eps}(a_{n,i}^{\eps})|=CM_n^{\eps}\eps^2.
	\end{equation*}
	It follows from \eqref{l357} that
	\begin{equation}\label{p334}
	s\sum_{n=0}^{N-1}\lt\lVert 1-|\rho_n^{\eps}|^q\rt\rVert_{L^r(\Omega)}^r\rightarrow 0
	\end{equation}
	for all $1\leq r <\infty$ and $0<q<\infty$ as $(\eps,s)\rightarrow (0,0)$. Recall that for all $n$, one can write $v_{n,1}=-2\hat{\nabla}\times\hat\Delta^{-1}_D w_n$, where $v_{n,1}$ is given in \eqref{vn}. Given $p<\frac{3}{2}$, we have
	\begin{equation*}
	\begin{split}
	\lt\lVert \sum_{n=0}^{N-1}(v_n^{\eps}-v_{n,1})\chi_n \rt\rVert_{L^p(D)}^p &= s\sum_{n=0}^{N-1}\lt\lVert v_n^{\eps}-v_{n,1} \rt\rVert_{L^p(\Omega)}^p\\
	&=s\sum_{n=0}^{N-1}\lt\lVert -2\hat{\nabla}\times\hat\Delta_D^{-1}(w_n^{\eps}-w_n) \rt\rVert_{L^p(\Omega)}^p.
	\end{split}
	\end{equation*}
	It follows from standard elliptic estimates and \eqref{l354} that
	\begin{equation}\label{p335}
	\lt\lVert \sum_{n=0}^{N-1}(v_n^{\eps}-v_{n,1})\chi_n \rt\rVert_{L^p(D)}^p \leq sC\sum_{n=0}^{N-1}\lt\lVert w_n^{\eps}-w_n \rt\rVert_{W^{-1,p}(\Omega)}^p\rightarrow 0
	\end{equation}
	as $(\epsilon,s)\rightarrow (0,0)$, where $C$ is a constant depending only on $\Omega$ and $p$. It is clear from Lemma \ref{l4.2} that
	\begin{equation*}
	\lt\lVert \sum_{n=0}^{N-1}v_{n,1}\chi_n -v_1\rt\rVert_{L^p(D)}\rightarrow 0
	\end{equation*}
	as $s\rightarrow 0$. Therefore we conclude that
	\begin{equation}\label{p337}
	\lt\lVert \sum_{n=0}^{N-1}v_n^{\eps}\chi_n -v_1\rt\rVert_{L^p(D)} \rightarrow 0.
	\end{equation}
	For any $p<r<\frac{3}{2}$, using H\"{o}lder's inequality, we have
	\begin{equation*}
	\lt\lVert \sum_{n=0}^{N-1}\lt(|\rho_n^{\epsilon}|^2-1\rt)v_n^{\eps}\chi_n \rt\rVert_{L^p(D)} \leq \lt\lVert\sum_{n=0}^{N-1}\lt(|\rho_n^{\epsilon}|^2-1\rt)\chi_n \rt\rVert_{L^{\tilde r}(D)}\  \lt\lVert\sum_{n=0}^{N-1}v_n^{\eps}\chi_n \rt\rVert_{L^r(D)},
	\end{equation*}
	where $\frac{1}{r}+\frac{1}{\tilde r}=\frac{1}{p}$. We deduce from \eqref{p334} and \eqref{p337} that
	\begin{equation}\label{p339}
	\lt\lVert \sum_{n=0}^{N-1}\lt(|\rho_n^{\epsilon}|^2-1\rt)v_n^{\eps}\chi_n \rt\rVert_{L^p(D)} \rightarrow 0
	\end{equation}
	as $(\epsilon,s)\rightarrow (0,0)$. Hence \eqref{p3.3.1} follows from \eqref{p338}, \eqref{p337} and \eqref{p339}, and \eqref{p3.3.2} is a direct consequence.
	
	\textit{Step 2. Decomposition of the energy:} By the definition of $u_n^{\eps}$, we have
	\begin{equation*}
	|\hat{\nabla}u_n^{\eps}|^2=|\hat\nabla \rho_n^{\eps}|^2 + (\rho_n^{\eps})^2|\hat\nabla\phi_n^{\eps}|^2 = |\hat\nabla \rho_n^{\eps}|^2 + |\ln\epsilon|^2 (\rho_n^{\eps})^2 |v_n^{\eps}|^2.
	\end{equation*}
	Therefore
	\begin{equation*}
	s\sum_{n=0}^{N-1}E_{\epsilon}(u_n^{\eps})=\frac{s}{2}\sum_{n=0}^{N-1}\int_{\Omega}\left[|\hat\nabla \rho_n^{\eps}|^2 + |\ln\epsilon|^2 (\rho_n^{\eps})^2 |v_n^{\eps}|^2+\frac{\lt(1-(\rho_n^{\eps})^2\rt)^2}{2\eps^2}\right] d\hat x.
	\end{equation*}
	Using the definition of $\rho_n^{\eps}$, \eqref{l341}-\eqref{l342} and \eqref{l357}, we have
	\begin{equation*}
	\begin{split}
	s\sum_{n=0}^{N-1}\int_{\Omega}\left[|\hat\nabla \rho_n^{\eps}|^2 +\frac{|1-(\rho_n^{\eps})^2|^2}{2\eps^2}\right] d\hat x &= s\sum_{n=0}^{N-1}\int_{\cup B_{\eps}(a_{n,i}^{\eps})}\left[|\hat\nabla \rho_n^{\eps}|^2 +\frac{|1-(\rho_n^{\eps})^2|^2}{2\eps^2}\right] d\hat x\\
	&\leq Cs\sum_{n=0}^{N-1} M_n^{\eps} \leq C|\ln\epsilon|\lVert w \rVert_{L^1(D)}.
	\end{split}
	\end{equation*}
	Therefore, we have
	\begin{equation}\label{p3311}
	\frac{s}{|\ln\epsilon|^2}\sum_{n=0}^{N-1}E_{\epsilon}(u_n^{\eps}) = \frac{s}{2}\sum_{n=0}^{N-1}\int_{\Omega}(\rho_n^{\eps})^2 |v_n^{\eps}|^2 d\hat x + o_{\eps}(1),
	\end{equation}
	where $\lim_{\eps\rightarrow 0}o_{\eps}(1)=0$. Let us write
	\begin{equation}\label{p3312}
	\begin{split}
	\frac{s}{2}\sum_{n=0}^{N-1}\int_{\Omega}(\rho_n^{\eps})^2 |v_n^{\eps}|^2 d\hat x = &\frac{s}{2}\sum_{n=0}^{N-1}\int_{\Omega}|v_{n,1}|^2 d\hat x + s\sum_{n=0}^{N-1}\int_{\Omega}v_{n,1}\cdot\lt(\rho_n^{\eps}v_n^{\eps}-v_{n,1}\rt) d\hat x\\
	& + \frac{s}{2}\sum_{n=0}^{N-1}\int_{\Omega}\lt|\rho_n^{\eps}v_n^{\eps}-v_{n,1}\rt|^2 d\hat x.
	\end{split}
	\end{equation}
	By Lemma \ref{l4.2}, we have
	\begin{equation}\label{p3313}
	\frac{s}{2}\sum_{n=0}^{N-1}\int_{\Omega}|v_{n,1}|^2 d\hat x \rightarrow \frac{\lVert v_1 \rVert_{L^2(D)}^2}{2}
	\end{equation}
	as $s\rightarrow 0$. By \eqref{p335} and a similar argument as in \eqref{p339}, one can show
	\begin{equation}\label{p3318}
	\sum_{n=0}^{N-1}\lt(\rho_n^{\eps}v_n^{\eps}-v_{n,1}\rt)\chi_n \rightarrow 0 \quad \text{ in } L^p(D;\mb R^2) \blue{\text{ for all } p<\frac{3}{2}}.
	\end{equation}
	In the following we show that 
	\begin{equation}\label{p3315}
	\frac{s}{2}\sum_{n=0}^{N-1}\int_{\Omega}|\rho_n^{\eps}v_n^{\eps}-v_{n,1}|^2 d\hat x \leq \lVert w \rVert_{L^1(D)} + o_{\eps,s}(1).
	\end{equation}
	The above implies that $\{\sum(\rho_n^{\eps}v_n^{\eps}-v_{n,1})\chi_n\}$ forms a bounded sequence in $L^2(D;\mb R^2)$. It follows from this and \eqref{p3318} that, up to a subsequence, $\sum\lt(\rho_n^{\eps}v_n^{\eps}-v_{n,1}\rt)\chi_n$ converges weakly to $0$ in $L^2$. Therefore we deduce from Lemma \ref{l4.2} that
	\begin{equation}\label{p3314}
	s\sum_{n=0}^{N-1}\int_{\Omega}v_{n,1}\cdot(\rho_n^{\eps}v_n^{\eps}-v_{n,1}) d\hat x \rightarrow 0.
	\end{equation}
	The upper bound \eqref{p3.3.3} then follows from \eqref{p3311}-\eqref{p3314}.
	
	\textit{Step 3. Proof of \eqref{p3315}:} As in the proof of Lemma 7.2 in \cite{JS}, we define
	\begin{equation*}
	\delta = \delta(\eps) = c_0|\ln\epsilon|^{-\frac{1}{2}}/3,
	\end{equation*}
	where $c_0$ is the constant in \eqref{l353}. Recall that we defined the function $\eta:\mathbb R^2\rightarrow \mathbb R$ to be the standard mollifier with $\int_{\mb R^2}\eta =1$. Using \blue{Young's inequality}, we have
	\begin{equation}\label{p3320}
	\begin{split}
	&s\sum_{n=0}^{N-1}\lVert \rho_n^{\eps}v_n^{\eps}-v_{n,1} \rVert_{L^2(\Omega)}^2 \\
	&\leq s\sum_{n=0}^{N-1}\bigg(2\lt(1+\frac{1}{\sigma}\rt)\lVert \rho_n^{\eps}v_n^{\eps}-\eta^{\eps}*v_n^{\eps} \rVert_{L^2(\Omega)}^2+\lt(1+\sigma\rt)\lVert \eta^{\eps}*v_n^{\eps}-\eta^{\delta}*v_n^{\eps} \rVert_{L^2(\Omega)}^2\\
	&\quad\quad\quad\quad\quad\quad\quad\quad\quad\quad\quad\quad\quad\quad\quad\quad\quad\quad+2\lt(1+\frac{1}{\sigma}\rt)\lVert \eta^{\delta}*v_n^{\eps}-v_{n,1} \rVert_{L^2(\Omega)}^2 \bigg)\\
	&= \blue{s\sum_{n=0}^{N-1}\left( 2\lt(1+\frac{1}{\sigma}\rt)A_n^{\eps}+\lt(1+\sigma\rt)B_n^{\eps}+2\lt(1+\frac{1}{\sigma}\rt)C_n^{\eps} \right)}
	\end{split}
	\end{equation}
	\blue{for all $\sigma>0$.} Here we use the convention that $\eta^{r}*v_n^{\eps}=v_n^{\eps}$ if $\text{dist}(\hat x,\partial\Omega)<r$.
	
	According to (7.17) in \cite{JS}, we have $\lVert \eta^{\delta}*v_n^{\eps} \rVert_{W^{1,q}(\Omega)} \leq C$ for all $q<\infty$ and for some constant $C$ independent of $\epsilon$ and $s$. Using the fact that $sN=L$ is fixed, we obtain
	\begin{equation}\label{p3321}
	\lt\lVert\sum_{n=0}^{N-1}(\eta^{\delta}*v_n^{\eps})\chi_n \rt\rVert_{L^q(D)}^q = s\sum_{n=0}^{N-1}\lVert\eta^{\delta}*v_n^{\eps}\rVert_{L^q(\Omega)}^q \leq C
	\end{equation}
	for some constant $C$ independent of $\eps$ and $s$. Since $w=\frac{1}{2}\cl v\in C^{\infty}(\overline D)$, we have $ w_n\in L^{\infty}(\Omega)$. Recall that $v_{n,1}=-2\hat{\nabla}\times\hat\Delta^{-1}_D w_n$. It follows from global elliptic regularity that $\lVert v_{n,1}\rVert_{W^{1,q}(\Omega)}\leq C$ for all $q<\infty$ and for some constant $C$ depending only on $\lVert w\rVert_{\infty}$, $\Omega$ and $q$. Therefore, we have
	\begin{equation}\label{p3322}
	\lt\lVert\sum_{n=0}^{N-1}v_{n,1}\chi_n \rt\rVert_{L^q(D)}^q\leq C.
	\end{equation}
	Given $p<\frac{3}{2}$, we have
	\begin{equation*}
	\begin{split}
	\bigg\lVert\sum_{n=0}^{N-1}&(\eta^{\delta}*v_n^{\eps}-v_{n,1})\chi_n \bigg\rVert_{L^p(D)}\\
	&\leq \lt\lVert\sum_{n=0}^{N-1}(\eta^{\delta}*v_n^{\eps}-v_n^{\eps})\chi_n \rt\rVert_{L^p(D)}+\lt\lVert\sum_{n=0}^{N-1}(v_n^{\eps}-v_{n,1})\chi_n \rt\rVert_{L^p(D)}.
	\end{split}
	\end{equation*}
	The second term in the above right-hand side converges to zero by \eqref{p335}. The first term also converges to zero by approximation to identity along with the compactness \eqref{p335}. Hence we have
	\begin{equation*}
	\lt\lVert\sum_{n=0}^{N-1}(\eta^{\delta}*v_n^{\eps}-v_{n,1})\chi_n \rt\rVert_{L^p(D)} \rightarrow 0.
	\end{equation*}
	Using \eqref{p3321} and \eqref{p3322}, it is clear that $\{\sum(\eta^{\delta}*v_n^{\eps}-v_{n,1})\chi_n\}$ is bounded in $L^{q}(D;\mb R^2)$ for all $q<\infty$. Therefore, we deduce from an interpolation inequality that
	\begin{equation}\label{p3316}
	s\sum_{n=0}^{N-1}C_n^{\eps}=\lt\lVert\sum_{n=0}^{N-1}(\eta^{\delta}*v_n^{\eps}-v_{n,1})\chi_n \rt\rVert_{L^2(D)}^2 \rightarrow 0.
	\end{equation}
	
	From the proof of Lemma 7.2 in \cite{JS}, we have
	\begin{equation*}
	B_n^{\eps}\leq \blue{2}\pi\frac{M_n^{\eps}}{|\ln\eps|^2}\left(\ln\left(\frac{\delta}{\eps}\right) +C\right),
	\end{equation*}
	and
	\begin{equation*}
	A_n^{\eps}\leq C\int_{\Omega}\lt(1-\rho_n^{\eps}\rt)^2\left(|\eta^{\delta}*v_n^{\eps}|^2+C|\ln\eps|^{-1}\right)d\hat x,
	\end{equation*}
	where the above constants $C$ are independent of $\eps$ and $s$. (Note that the corresponding estimate for $B_n^{\eps}$ in \cite{JS} is off by a factor of $2$. However, this is not serious. The arguments there can be corrected by using Young's inequality as we have in (\ref{p3320}) and letting $\sigma\rightarrow 0$.) Therefore, noting that $\ln\delta<0$ and using \eqref{l357}, we have
	\begin{equation}\label{p3317}
	\begin{split}
	s\sum_{n=0}^{N-1}B_n^{\eps}&\leq s\sum_{n=0}^{N-1}\blue{2}\pi\frac{M_n^{\eps}}{|\ln\eps|^2}\left(\ln\left(\frac{\delta}{\eps}\right) +C\right)\\
	&\leq s\sum_{n=0}^{N-1}\blue{2}\pi\frac{M_n^{\eps}}{|\ln\eps|^2}\left(|\ln\eps| +C\right) \leq \blue{2}\lVert w\rVert_{L^1(D)}+o_{\eps,s}(1).
	\end{split}
	\end{equation}
	Using H\"{o}lder's and the Cauchy-Schwarz inequalities, we have
	\begin{equation*}
	\begin{split}
	s\sum_{n=0}^{N-1}A_n^{\eps}&\leq C\left(s\sum_{n=0}^{N-1}\int_{\Omega}(1-\rho_n^{\eps})^4d\hat x\right)^{\frac{1}{2}}\left(s\sum_{n=0}^{N-1}\int_{\Omega}|\eta^{\delta}*v_n^{\eps}|^4d\hat x\right)^{\frac{1}{2}}\\
	& + \frac{C}{|\ln\eps|}s\sum_{n=0}^{N-1}\int_{\Omega}(1-\rho_n^{\eps})^2 d\hat x.
	\end{split}
	\end{equation*}
	We deduce from \eqref{p334} and \eqref{p3321} that
	\begin{equation}\label{p3319}
	s\sum_{n=0}^{N-1}A_n^{\eps} \rightarrow 0.
	\end{equation}
	\blue{Finally, putting \eqref{p3316}-\eqref{p3319} into \eqref{p3320}, and first letting $(\epsilon,s)\rightarrow(0,0)$ and then letting $\sigma\rightarrow 0$, we obtain \eqref{p3315}.}
\end{proof}

\subsection{Proof of Theorem \ref{p3.3} completed}

\begin{lemma}\label{l4.6}
	Let $v_2$ be as in Lemma \ref{l4.2}. There exists a sequence $\{u_n^{\epsilon}\}_{n=0}^{N}\subset [H^1(\Omega;\mb C)]^{N+1}$ such that
	\begin{equation}\label{l4.6.1}
	\frac{1}{|\ln\epsilon|}\sum_{n=0}^{N-1}j(u_n^{\epsilon})\chi_n \rightarrow v_2 \quad\text{in } L^2(D;\mathbb R^2)
	\end{equation}
	and
	\begin{equation}\label{l4.6.2}
	s\sum_{n=0}^{N}\frac{E_{\epsilon}(u_n^{\epsilon})}{|\ln\epsilon|^2}\rightarrow \frac{\lVert v_2\rVert_{L^2(D)}^2}{2}
	\end{equation}
	as $(\eps,s)\rightarrow (0,0)$.
\end{lemma}

\begin{proof}
	Since $v_{n,2}\in\mathscr G$ for all $n$, we have $v_{n,2} = \hat{\nabla}g_n$ for some $g_n\in H^1(\Omega)$ and $\cl v_{n,2}=0$. Define $u_n^{\eps} = e^{i|\ln\eps|g_n}$. By simple calculations, we have
	\begin{equation}\label{l462}
	|\hat{\nabla} u_n^{\eps}|^2 = |\ln\eps|^2|\hat{\nabla} g_n|^2 = |\ln\eps|^2|v_{n,2}|^2
	\end{equation}
	and
	\begin{equation}\label{l461}
	j(u_n^\eps) = |\ln\eps|v_{n,2}.
	\end{equation}
	It follows from Lemma \ref{l4.2} that
	\begin{equation*}
	\frac{1}{|\ln\epsilon|}\sum_{n=0}^{N-1}j(u_n^{\epsilon})\chi_n = v^s_2\rightarrow v_2 \quad\text{in } L^2(D;\mathbb R^2).
	\end{equation*}
	Using \eqref{l462} and the fact that $|u_n^{\eps}|=1$, we have
	\begin{equation*}
	E_{\eps}(u_n^{\eps}) = \frac{|\ln\eps|^2}{2}\lVert v_{n,2} \rVert_{L^2(\Omega)}^2.
	\end{equation*}
	It follows from Lemma \ref{l4.2} again that \eqref{l4.6.2} holds.
\end{proof}

\begin{proof}[Proof of Theorem \ref{p3.3}]
	Let $\{u_{n,1}^{\eps}\}$ and $\{u_{n,2}^{\eps}\}$ be the sequences in Lemmas \ref{l4.3} and \ref{l4.6}, respectively. Define $u_n^{\eps}=u_{n,1}^{\eps}u_{n,2}^{\eps}$. We show that the sequence $\{u_n^{\eps}\}$ satisfies the conclusions of Theorem \ref{p3.3}.
	
	First we show compactness. By direct calculations, we have $j(u_n^{\eps}) = j(u_{n,1}^{\eps}) + j(u_{n,2}^{\eps})|u_{n,1}^{\eps}|^2$. By Lemma \ref{l4.3}, we have
	\begin{equation}\label{l481}
	\frac{1}{|\ln\epsilon|}\sum_{n=0}^{N-1}j(u_{n,1}^{\epsilon})\chi_n\rightarrow v_1 \quad\text{in } L^p(D;\mathbb R^2) \text{ for all } p<\frac{3}{2}.
	\end{equation}
	Using \eqref{l461}, we have
	\begin{equation*}
	\frac{1}{|\ln\epsilon|}\sum_{n=0}^{N-1}j(u_{n,2}^{\epsilon})|u_{n,1}^{\eps}|^2\chi_n - \sum_{n=0}^{N-1}v_{n,2}\chi_n=\sum_{n=0}^{N-1} v_{n,2}(|u_{n,1}^{\eps}|^2-1)\chi_n.
	\end{equation*}
	Therefore, given $p<\frac{3}{2}$, using H\"{o}lder's inequality, we obtain
	\begin{equation*}
	\begin{split}
	\bigg\lVert\frac{1}{|\ln\epsilon|}\sum_{n=0}^{N-1}j(u_{n,2}^{\epsilon})|u_{n,1}^{\eps}|^2\chi_n &- \sum_{n=0}^{N-1}v_{n,2}\chi_n\bigg\rVert_p^p = s\sum_{n=0}^{N-1}\lt\lVert v_{n,2}(|u_{n,1}^{\eps}|^2-1)\rt\rVert_p^p\\
	&\leq s\sum_{n=0}^{N-1} \lt\lVert v_{n,2}\rt\rVert_2^p\cdot \lt\lVert 1-|u_{n,1}^{\eps}|^2 \rt\rVert_{\frac{2p}{2-p}}^p\\
	&\leq \lt( s\sum_{n=0}^{N-1} \lt\lVert v_{n,2}\rt\rVert_2^2\rt)^{\frac{p}{2}} \lt( s\sum_{n=0}^{N-1} \lt\lVert 1-|u_{n,1}^{\eps}|^2\rt\rVert_{\frac{2p}{2-p}}^{\frac{2p}{2-p}}\rt)^{\frac{2-p}{2}}.
	\end{split}
	\end{equation*}
	It follows from Lemma \ref{l4.2} and \eqref{p334} that the above right side converges to zero, and therefore, we have
	\begin{equation}\label{l482}
	\frac{1}{|\ln\epsilon|}\sum_{n=0}^{N-1}j(u_{n,2}^{\epsilon})|u_{n,1}^{\eps}|^2\chi_n \rightarrow v_2 \quad\text{in } L^p(D;\mathbb R^2) \text{ for all } p<\frac{3}{2}.
	\end{equation}
	Combining \eqref{l481} with \eqref{l482} and using $v_1+v_2=v$, we have \eqref{p3.3.1}, and \eqref{p3.3.2} is an immediate consequence.
	
	To show \eqref{p3.3.3}, by direct computations using the fact $|u_{n,2}^{\eps}|=1$, we have
	\begin{equation*}
	\begin{split}
	|\hat{\nabla}u_n^{\eps}|^2 &= |\hat{\nabla}u_{n,1}^{\eps}|^2 + |u_{n,1}^{\eps}|^2 |\hat{\nabla}u_{n,2}^{\eps}|^2 + 2j(u_{n,1}^{\eps})\cdot j(u_{n,2}^{\eps})\\
	& \leq |\hat{\nabla}u_{n,1}^{\eps}|^2 + |\hat{\nabla}u_{n,2}^{\eps}|^2 + 2j(u_{n,1}^{\eps})\cdot j(u_{n,2}^{\eps}).
	\end{split}
	\end{equation*}
	Therefore, using $|u_{n,2}^{\eps}|=1$ again, we have 
	\begin{equation}\label{l484}
	s\sum_{n=0}^{N}\frac{E_{\epsilon}(u_n^{\epsilon})}{|\ln\epsilon|^2}\leq s\sum_{n=0}^{N}\frac{E_{\epsilon}(u_{n,1}^{\epsilon})}{|\ln\epsilon|^2} + s\sum_{n=0}^{N}\frac{E_{\epsilon}(u_{n,2}^{\epsilon})}{|\ln\epsilon|^2} + s\sum_{n=0}^{N}\int_{\Omega}\frac{j(u_{n,1}^{\eps})}{|\ln\eps|}\cdot \frac{j(u_{n,2}^{\eps})}{|\ln\eps|} d\hat x.
	\end{equation}
	Since $|j(u_{n,1}^{\eps})|\leq |\hat{\nabla}u_{n,1}^{\eps}|$, it follows from Lemma \ref{l4.3} that the sequence $\{\sum\frac{j(u_{n,1}^{\epsilon})}{|\ln\epsilon|}\chi_n\}$ is bounded in $L^2(D;\mb R^2)$. Since the sequence converges strongly to $v_1$ in $L^p(D)$ for $p<\frac{3}{2}$, we conclude that
	\begin{equation}\label{l483}
	\sum_{n=0}^{N-1}\frac{j(u_{n,1}^{\epsilon})}{|\ln\epsilon|}\chi_n \rightharpoonup v_1 \quad \text{ in } L^2(D;\mb R^2).
	\end{equation}
	We deduce from \eqref{l483}, \eqref{l4.6.1} and Lemma \ref{l4.2} that
	\begin{equation}\label{l485}
	s\sum_{n=0}^{N}\int_{\Omega}\frac{j(u_{n,1}^{\eps})}{|\ln\eps|}\cdot \frac{j(u_{n,2}^{\eps})}{|\ln\eps|} d\hat x \rightarrow \int_{D} v_1\cdot v_2 dx = 0.
	\end{equation}
	Finally, the estimate \eqref{p3.3.3} follows from \eqref{l484}, \eqref{l485} and Lemmas \ref{l4.3} and \ref{l4.6}, and the fact that $\lVert v\rVert_{L^2}^2 = \lVert v_1\rVert_{L^2}^2+\lVert v_2\rVert_{L^2}^2$.
\end{proof}

\section{Construction of the magnetic potential}

In this section, we construct the magnetic potential for the recovery sequence. Given $(v, \vec A) \in V\times K_0$, define
\begin{equation}\label{4.1}
\vec A^{\epsilon}(x) = |\ln\epsilon| \vec A(x)+\lt(h_{ex}-h_0|\ln\eps|\rt)\vec a(x)\in K,
\end{equation}
so that
\begin{equation*}
\frac{\vec A^{\epsilon}-h_{ex}\vec a}{|\ln\epsilon|} = \vec A - h_0 \vec a
\end{equation*}
and
\begin{equation}\label{3.2}
\frac{1}{2|\ln\epsilon|^2}\int_{\mathbb R^3} \left| \nabla\times \vec A^{\epsilon} - h_{ex}\vec e_3\right|^2 dx = \frac{1}{2}\int_{\mathbb R^3} \left| \nabla\times \vec A - h_{0}\vec e_3\right|^2 dx.
\end{equation}
The main result of this section is the following estimate:

\begin{lemma}\label{prop4.7}
	Given $\vec A \in K_0$, let $\vec A^{\eps}$ be defined as in \eqref{4.1}. Then we have
	\begin{equation}\label{prop4.7.0}
	\sum_{n=0}^{N-1}\frac{\hat A_n^{\eps}}{|\ln\eps|}\chi_n \rightarrow \hat A \quad\text{ in } L^2(D;\mb R^2) \text{ as } (\epsilon,s)\rightarrow (0,0),
	\end{equation}
	where $\hat A_n^{\eps}$ is the trace of $\hat A^{\eps}$ on $\Omega\times\{ns\}$.
\end{lemma}

\begin{proof}
	By the definition of $\vec A^{\eps}$ in \eqref{4.1}, we have
	\begin{equation*}
	\frac{\hat A_n^{\eps}}{|\ln\eps|} = \hat A_{n}+\lt(\frac{h_{ex}}{|\ln\eps|}-h_0\rt)\hat a_n,
	\end{equation*}
	where $\hat A_n^{\eps}$, $\hat A_{n}$ and $\hat a_n$ are the traces of the corresponding functions on $\Omega\times\{ns\}$. Therefore, by the triangle inequality, we have
	\begin{equation*}
	\begin{split}
	\bigg\lVert\sum_{n=0}^{N-1}&\frac{\hat A_n^{\eps}}{|\ln\eps|}\chi_n - \hat A \bigg\rVert_{L^2(D)} \\
	&\leq \lt\lVert\sum_{n=0}^{N-1}\hat A_{n}\chi_n - \hat A \rt\rVert_{L^2(D)} + \lt\lVert\sum_{n=0}^{N-1}\lt(\frac{h_{ex}}{|\ln\eps|}-h_0\rt)\hat a_n\chi_n \rt\rVert_{L^2(D)}.
	\end{split}
	\end{equation*}
	Since $\vec a$ is smooth on $\mathbb{R}^3$, in particular, it is bounded on $\overline{D}$. Using $\lim_{\eps\rightarrow 0}\frac{h_{ex}}{|\ln\eps|}=h_0$, it is clear that
	\begin{equation*}
	\lt\lVert\sum_{n=0}^{N-1}\lt(\frac{h_{ex}}{|\ln\eps|}-h_0\rt)\hat a_n\chi_n \rt\rVert_{L^2(D)} \rightarrow 0 \quad\text{ as } \epsilon\rightarrow 0.
	\end{equation*}
	It follows from Lemma \ref{l2.2} that
	\begin{equation*}
	\begin{split}
	\lt\lVert\sum_{n=0}^{N-1}\hat A_{n}\chi_n - \hat A \rt\rVert_{L^2(D)}^2 &= \sum_{n=0}^{N-1}\int_{ns}^{(n+1)s}\int_{\Omega}\lt|\hat A_n - \hat A\rt|^2 d\hat xdx_3\\
	& \leq s^2\int_{D}\lt| \nabla\vec{A}\rt|^2dx \rightarrow 0 \text{ as } s\rightarrow 0.
	\end{split}
	\end{equation*}
	The estimate (\ref{prop4.7.0}) follows from the above estimates.
\end{proof}

\section{Proof of Theorem \ref{T1}: completed}

In this section, we complete the proof of the upper bound in Theorem \ref{T1}. Then we give the proof of Corollary \ref{cor1} in the introduction.

\begin{proof}[Proof of \eqref{thm1.4}]
	Let $(v,\vec A)\in V\times K_0$. Using Proposition \ref{prop2.2}, we can find a sequence $\{v_k\}\subset V\cap C^{\infty}(\overline{D};\mb R^2)$ such that \eqref{prop2.2.1}-\eqref{prop2.2.2} hold. Using Theorem \ref{p3.3} for each $v_k$ and a diagonal argument, we can find $\{u_n^{\eps}\}$ such that
	\begin{equation*}
	v^{\epsilon,s}:=\frac{1}{|\ln\epsilon|}\sum_{n=0}^{N-1}j(u_n^{\epsilon})\chi_n\rightarrow v \quad\text{in } L^p(D;\mathbb R^2) \text{ for all } p<\frac{3}{2},
	\end{equation*} 
	\begin{equation*}
	w^{\epsilon,s}:=\frac{1}{|\ln\epsilon|}\sum_{n=0}^{N-1}J(u_n^{\epsilon})\chi_n\rightharpoonup w \quad\text{in } W^{-1,p}(D) \text{ for all } p<\frac{3}{2},
	\end{equation*}
	and
	\begin{equation}\label{t13}
	\limsup_{(\eps,s)\rightarrow(0,0)} s\sum_{n=0}^{N}\frac{E_{\epsilon}(u_n^{\epsilon})}{|\ln\epsilon|^2}\leq \frac{\lVert v \rVert_{L^2(D)}^2}{2} + |w|(D),
	\end{equation}
	where $w=\frac{1}{2}\cl v$. Define $\vec A^{\eps}$ as in \eqref{4.1}. By the energy upper bound \eqref{t13} and Theorem \ref{lem2.3}, it is clear that the configurations $(\{u_n^{\eps}\},\vec A^{\eps})$ satisfy the compactness estimates in Theorem \ref{T1}. 
	
	Next let us calculate the energy $\mathcal{G}_{LD}^{\eps,s}(\{u_n^{\eps}\},\vec A^{\eps})$ using the decomposition formula \eqref{splitting}. From \eqref{thm1.1} and Lemma \ref{prop4.7}, we have
	\begin{equation*}
	\sum\limits_{n=0}^{N-1}\frac{j(u_n^{\eps})}{|u_n^{\eps}||\ln\eps|}\chi_n \rightharpoonup v \text{ in } L^2(D;\mathbb{R}^2) \quad\text{and}\quad \sum_{n=0}^{N-1}\frac{\hat A_n^{\eps}}{|\ln\eps|}\chi_n \rightarrow \hat A \text{ in } L^2(D;\mb R^2).
	\end{equation*}
	Using exactly the same arguments that were used to establish (\ref{2.4.9}), we obtain from the above compactness that
	\begin{equation}\label{t24}
	\frac{s}{|\ln\eps|^2}\sum_n \int_\Omega(\hat{\nabla}u_n^{\eps},\imath u_n^{\eps})\cdot\hat{A}_n^{\eps}d\hat x \rightarrow \int_{D}v\cdot\hat Adx.
	\end{equation}
	Next, using $|u_n^{\eps}|\leq 1$ and Lemma \ref{prop4.7}, we have
	\begin{equation}\label{t16}
	\begin{split}
	\limsup_{(\eps,s)\rightarrow(0,0)} \frac{s}{|\ln\eps|^2}&\sum_n \int_\Omega\frac{1}{2}|u_n^{\eps}|^2|\hat{A}_n^{\eps}|^2d\hat x\\
	&\leq \lim_{(\eps,s)\rightarrow(0,0)} \frac{s}{|\ln\eps|^2}\sum_n \int_\Omega\frac{1}{2}|\hat{A}_n^{\eps}|^2d\hat x = \frac{1}{2}\lVert \hat A\rVert_{L^2(D)}^2.
	\end{split}
	\end{equation}
	Using $|u_n^{\eps}|\leq 1$ and the assumption \eqref{h}, it is clear that
	\begin{equation}\label{t22}
	\frac{s}{|\ln\eps|^2}\sum^{N-1}_{n=0}\int_\Omega\frac{1}{2\lambda^{2}s^2}\left| u_{n+1}^{\eps}-u_{n}^{\eps}e^{\imath \int_{ns}^{(n+1)s}(A^{\eps})^{3}dx_{3}}\right|^{2}d\hat x \rightarrow 0.
	\end{equation}
	Finally, plugging \eqref{t13}-\eqref{t22} and \eqref{3.2} into \eqref{splitting}, we obtain \eqref{thm1.4}. This completes the proof of Theorem \ref{T1}.
\end{proof}

Finally, we prove Corollary \ref{cor1} in Section 1. \blue{Note that here we restrict our attention to minimizers. From \cite{BK}, minimizers of the Lawrence-Doniach energy satisfy $|u_n|\leq 1$ a.e. in $\Omega$.}

\begin{proof}[Proof of Corollary \ref{cor1}]
	From the Euler-Lagrange equations \eqref{LDsystem}, we have
	\begin{equation*}
	\nabla\times(\nabla\times\vec{A}^{\epsilon,s})=(j_1,j_2,j_3) \quad\text{ in } \mathbb{R}^3,
	\end{equation*}
	where $j_i$ for $i=1,2,3$ are given in \eqref{j}. Using the fact $|u_n^{\eps}|\leq 1$ and the assumption \eqref{h}, we have
	\begin{equation}\label{cor1.1}
	\begin{split}
	\lt\lVert\frac{j_3}{|\ln\eps|}\rt\rVert_{L^2(\mb R^3)}^2&\leq \sum_{n=0}^{N-1}\int_{ns}^{(n+1)s}\int_{\Omega}\frac{s^2}{|\ln\eps|^2}\frac{1}{\lambda^4s^4}\lt|u_{n+1}^{\eps}-u_{n}^{\eps}e^{\imath \int_{ns}^{(n+1)s}(A^{\epsilon,s})^{3}dx_{3}}\rt|^2d\hat xdx_3\\
	&\leq CNs\cdot\frac{s^2}{|\ln\eps|^2}\frac{1}{\lambda^4s^4}=\frac{CL}{\lambda^4}\cdot\frac{1}{s^2|\ln\eps|^2}\rightarrow 0.
	\end{split}
	\end{equation}
	For $i=1,2$, define
	\begin{equation*}
	\begin{split}
	\tilde j_i &:= \sum_{n=0}^{N-1}(\partial_{i}u_{n}^{\eps}-\imath (A_n^{\epsilon,s})^iu_n^{\eps},\imath u_n^{\eps})\chi_n(x_3)\\
	&=\sum\limits_{n=0}^{N-1}(\partial_{i}u_{n}^{\eps},\imath u_n^{\eps})\chi_n(x_3) - \sum\limits_{n=0}^{N-1}(A_n^{\epsilon,s})^i|u_n^{\eps}|^2\chi_n(x_3).
	\end{split}
	\end{equation*}
	Let $\varphi\in C_c(\mb R^3)$ be a test function. It follows from the definitions of $j_i$ and $\tilde j_i$ that
	\begin{equation}\label{cor1.5}
	\begin{split}
	\int_{\mb R^3}&\varphi d(j_i-\tilde j_i)\\
	&=\sum_{n=0}^{N-1}\int_{ns}^{(n+1)s}\int_{\Omega}(\partial_{i}u_{n}^{\eps}-\imath (A_n^{\epsilon,s})^iu_n^{\eps},\imath u_n^{\eps})\lt(\varphi(\hat x,ns)-\varphi(\hat x,x_3)\rt)d\hat x dx_3\\
	&+s\int_{\Omega}(\partial_{i}u_{N}^{\eps}-\imath (A_N^{\epsilon,s})^iu_N^{\eps},\imath u_N^{\eps})\varphi(\hat x,Ns)d\hat x = I + II.
	\end{split}
	\end{equation}
	Using $|u_n^{\eps}|\leq 1$ and the uniform continuity of $\varphi$, we have
	\begin{equation*}
	\begin{split}
	|I| &\leq \sum_{n=0}^{N-1}\int_{ns}^{(n+1)s}\int_{\Omega}\lt|\partial_{i}u_{n}^{\eps}-\imath (A_n^{\epsilon,s})^iu_n^{\eps}\rt|\lt|\varphi(\hat x,ns)-\varphi(\hat x,x_3)\rt|d\hat x dx_3\\
	&\leq o_s(1) s\sum_{n=0}^{N-1}\int_{\Omega}\lt|\partial_{i}u_{n}^{\eps}-\imath (A_n^{\epsilon,s})^iu_n^{\eps}\rt|d\hat x\\
	&\leq o_s(1)C(\Omega) s\sum_{n=0}^{N-1}\lt(\int_{\Omega}\lt|\partial_{i}u_{n}^{\eps}-\imath (A_n^{\epsilon,s})^iu_n^{\eps}\rt|^2d\hat x\rt)^{\frac{1}{2}}\\
	&\leq o_s(1)C(\Omega) s^{\frac{1}{2}}\lt(s\sum_{n=0}^{N-1}\int_{\Omega}\lt|\partial_{i}u_{n}^{\eps}-\imath (A_n^{\epsilon,s})^iu_n^{\eps}\rt|^2d\hat x\rt)^{\frac{1}{2}},
	\end{split}
	\end{equation*}
	where in the last two steps in the above, we have used H\"{o}lder's and the Cauchy-Schwarz inequalities. Here $o_s(1)\rightarrow 0$ as $s\rightarrow 0$ and depends only on $\varphi$. From Theorem \ref{T1}, it is clear that the energy upper bound (\ref{thm1.5}) is satisfied by minimizers. Noticing that $\partial_{i}u_{n}^{\eps}-\imath (A_n^{\epsilon,s})^iu_n^{\eps}$ is the $i$th component of $\hat{\nabla}_{\hat A_n^{\epsilon,s}}u_n^{\eps}$, we obtain from the above that
	\begin{equation}\label{cor1.2}
	\frac{|I|}{|\ln\eps|} \leq o_s(1)Cs^{\frac{1}{2}} \rightarrow 0
	\end{equation}
	as $(\eps,s)\rightarrow (0,0)$. Similarly, using \eqref{thm1.5}, we have
	\begin{equation}\label{cor1.6}
	\begin{split}
	\frac{|II|}{|\ln\eps|}&\leq \sup|\varphi|\frac{s}{|\ln\eps|}\int_{\Omega}\lt|\partial_{i}u_{N}^{\eps}-\imath (A_N^{\epsilon,s})^iu_N^{\eps}\rt| d\hat x\\
	&\leq \sup|\varphi|\frac{s^{\frac{1}{2}}}{|\ln\eps|}\lt(s\int_{\Omega}\lt|\partial_{i}u_{N}^{\eps}-\imath (A_N^{\epsilon,s})^iu_N^{\eps}\rt|^2 d\hat x\rt)^{\frac{1}{2}}\leq C\sup|\varphi| s^{\frac{1}{2}}\rightarrow 0.
	\end{split}
	\end{equation}
	Putting \eqref{cor1.5}-\eqref{cor1.6} together, we deduce that
	\begin{equation}\label{cor1.7}
	\int_{\mb R^3}\varphi d\frac{j_i-\tilde j_i}{|\ln\eps|}\rightarrow 0.
	\end{equation}
	Using Theorem \ref{T1}, up to a subsequence, we have
	\begin{equation}
	\sum_{n=0}^{N-1}\frac{(\imath u_n^{\eps},\hat{\nabla}u_n^{\eps})}{|\ln\eps|}\chi_n(x_3)\rightharpoonup v_0 \quad\text{ in } L^{\frac{4}{3}}(D;\mb R^2)
	\end{equation}
	and 
	$$\frac{\vec A^{\epsilon,s} - h_{ex}\vec a}{|\ln\epsilon|} \rightharpoonup \vec A_0-h_0\vec a \quad\text{ in } \check H^1(\mathbb R^3;\mathbb R^3)$$
	for some $(v_0,\vec{A}_0)$. In particular, it is a standard consequence of $\Gamma$-convergence that $(v_0,\vec{A}_0)$ is a minimizer of the limiting functional $\mathcal{G}_{h_0}$. It follows from the upper bound \eqref{thm1.5} again that 
	\begin{equation*}
	s\sum_{n=0}^{N-1}\int_{\Omega}\lt(1-|u_n^{\eps}|^2\rt)^2d\hat x\leq C\eps^2|\ln\eps|^2 \rightarrow 0.
	\end{equation*}
	Using the above and arguments similar to those in the proof of \eqref{2.4.15}, one can show that, up to a subsequence,
	\begin{equation}\label{cor1.3}
	\sum\limits_{n=0}^{N-1}\frac{(A_n^{\epsilon,s})^i|u_n^{\eps}|^2}{|\ln\eps|}\chi_n(x_3)\rightarrow A_0^i \quad\text{ in } L^2(D).
	\end{equation}
	Using the definition of $\tilde j_i$ and \eqref{cor1.7}-\eqref{cor1.3}, we immediately obtain that, up to a subsequence,
	\begin{equation*}
	\lt|\int_{\mb R^3}\varphi d\frac{j_i}{|\ln\eps|}-\int_{D}\varphi(v_0^i-A_0^i)dx\rt|\rightarrow 0,
	\end{equation*}
	from which we conclude
	\begin{equation}\label{cor1.4}
	\frac{j_i}{|\ln\eps|}\rightharpoonup (v_0^i-A_0^i)\chi_D \quad\text{ in } \mathcal{M}(\mb R^3).
	\end{equation}
	Corollary \ref{cor1} then follows from \eqref{cor1.1} and \eqref{cor1.4}.
\end{proof}

\appendix
\section{Proof of Proposition \ref{prop2.2}}

In this appendix, we prove Proposition \ref{prop2.2}. We begin with the following local approximation result, whose proof adapts that for BV functions (see, e.g., \cite{EG}):

\begin{lemma}\label{l6.1}
	Assume $v\in V$. There exists a sequence $\{v_k\}_{k=1}^{\infty} \subset V\cap C^{\infty}(D;\mb R^2)$ such that
	\begin{equation}\label{l321}
	v_k\rightarrow v \quad\text{ in } L^2(D;\mathbb R^2)
	\end{equation}
	and 
	\begin{equation}\label{l322}
	|\text{curl}v_k|(D)\rightarrow |\text{curl}v|(D)
	\end{equation}
	as $k\rightarrow \infty$.
\end{lemma}

\begin{proof}
	Fix $\epsilon>0$. Given some positive integer $m$, define
	\begin{equation*}
	D_k = \{x\in D: \text{dist}(x,\partial D)>\frac{1}{m+k}\}
	\end{equation*}
	for $k\in\mathbb N$. We may choose $m$ sufficiently large such that 
	\begin{equation}\label{l323}
	|\cl v|(D\setminus D_1)<\epsilon.
	\end{equation}
	Denote $U_k=D_{k+1}\setminus \overline D_{k-1}$ with the convention $U_1 = D_2$. Let the sequence $\{\zeta_k\}$ be a partition of unity subordinate to $\{U_k\}$, i.e., 
	\begin{equation}\label{l328}
	\begin{cases}
	\zeta_k\in C^{\infty}_{c}(U_k),\\
	\sum_{k=1}^{\infty} \zeta_k = 1 \text{ on } D.
	\end{cases}
	\end{equation}
	Let $\eta$ be the standard mollifier. For each $k$, choose $\epsilon_k>0$ sufficiently small such that
	\begin{equation}\label{l324}
	\begin{cases}
	\eta_{\epsilon_k}*(v\zeta_k)\in C_{c}^{\infty}(U_k),\\
	\lt\lVert \eta_{\epsilon_k}*(v\zeta_k)-v\zeta_k\rt\rVert_{L^2(U_k;\mathbb R^2)} < \frac{\epsilon}{2^k},\\
	\lVert\eta_{\epsilon_k}*(v\cdot\hat\nabla^{\perp}\zeta_k)-v\cdot\hat\nabla^{\perp}\zeta_k\rVert_{L^1(U_k)}< \frac{\epsilon}{2^k}.
	\end{cases}
	\end{equation}
	
	Define
	\begin{equation}\label{l327}
	v_{\epsilon}=\sum_{k=1}^{\infty}\eta_{\epsilon_k}*(v\zeta_k).
	\end{equation}
	It is clear that $v_{\epsilon}\in C^{\infty}(D)$. By \eqref{l324}, we have
	\begin{equation}\label{l325}
	\lVert v-v_{\epsilon} \rVert_{L^2(D;\mathbb R^2)} < \epsilon \rightarrow 0
	\end{equation}
	as $\epsilon \rightarrow 0$. Let $\varphi\in C_{c}^1(D)$ be a test function such that $\sup|\varphi|\leq 1$. It follows from \eqref{l325} that
	\begin{equation*}
	-\int_{D}v\cdot\hat\nabla^{\perp}\varphi = -\lim_{\epsilon\rightarrow 0} \int_{D} v_{\epsilon}\cdot\hat\nabla^{\perp}\varphi\leq \liminf_{\epsilon\rightarrow 0}|\cl v_{\epsilon}|(D).
	\end{equation*}
	Hence we have
	\begin{equation}\label{l326}
	|\cl v|(D) \leq \liminf_{\epsilon\rightarrow 0}|\cl v_{\epsilon}|(D).
	\end{equation}
	
	Next, using the definition of $v_{\epsilon}$ in \eqref{l327}, we have
	\begin{equation}\label{l329}
	\begin{split}
	-\int_{D}v_{\epsilon}&\cdot\hat\nabla^{\perp}\varphi = -\sum_{k=1}^{\infty}\int_{D}\eta_{\epsilon_k}*(v\zeta_k)\cdot\hat\nabla^{\perp}\varphi\\
	&=-\sum_{k=1}^{\infty}\int_D v\zeta_k \cdot\hat\nabla^{\perp}(\eta_{\epsilon_k}*\varphi)\\
	&=-\sum_{k=1}^{\infty}\int_D v\cdot\hat\nabla^{\perp}(\zeta_k (\eta_{\epsilon_k}*\varphi)) + \sum_{k=1}^{\infty}\int_D v\cdot\hat\nabla^{\perp}\zeta_k(\eta_{\epsilon_k}*\varphi)\\
	&=-\sum_{k=1}^{\infty}\int_D v\cdot\hat\nabla^{\perp}\lt(\zeta_k (\eta_{\epsilon_k}*\varphi)\rt) + \sum_{k=1}^{\infty}\int_D \lt(\eta_{\epsilon_k}*(v\cdot\hat\nabla^{\perp}\zeta_k)-v\cdot\hat\nabla^{\perp}\zeta_k\rt)\varphi\\
	&= I_1^{\epsilon} + I_2^{\epsilon}.
	\end{split}\end{equation}
	Here we have used $\sum_{k=1}^{\infty}\hat\nabla^{\perp}\zeta_k = \hat\nabla^{\perp}(\sum_{k=1}^{\infty}\zeta_k) = 0$, which is a consequence of the definition of $\zeta_k$ in \eqref{l328}. By \eqref{l324}, we have $I_2^{\epsilon}\rightarrow 0$ as $\epsilon\rightarrow 0$. For $I_1^{\epsilon}$, we have
	\begin{equation*}
	-\sum_{k=1}^{\infty}\int_D v\cdot\hat\nabla^{\perp}(\zeta_k (\eta_{\epsilon_k}*\varphi)) = -\int_D v\cdot\hat\nabla^{\perp}(\zeta_1 (\eta_{\epsilon_1}*\varphi)) - \sum_{k=2}^{\infty}\int_D v\cdot\hat\nabla^{\perp}(\zeta_k (\eta_{\epsilon_k}*\varphi)). 
	\end{equation*}
	Note that, since $\sup|\varphi|\leq 1$, we have $|\zeta_k(\eta_{\epsilon_k}*\varphi)|\leq 1$. Therefore, we have
	\begin{equation*}
	-\int_D v\cdot\hat\nabla^{\perp}(\zeta_1 (\eta_{\epsilon_1}*\varphi)) \leq |\cl v|(D) \quad\text{and}\quad -\int_D v\cdot\hat\nabla^{\perp}(\zeta_k (\eta_{\epsilon_k}*\varphi)) \leq |\cl v|(U_k).
	\end{equation*}
	Note that each point in $D$ belongs to at most three sets of $\{U_k\}_{k=1}^{\infty}$. Hence, using \eqref{l323}, we have
	\begin{equation}\label{l3210}
	\begin{split}
	-\sum_{k=1}^{\infty}\int_D v\cdot\hat\nabla^{\perp}&(\zeta_k (\eta_{\epsilon_k}*\varphi)) \leq |\cl v|(D) + \sum_{k=2}^{\infty} |\cl v|(U_k)\\
	&\leq |\cl v|(D) + 3 |\cl v|(D\setminus D_1) \leq |\cl v|(D) + 3\epsilon.
	\end{split}
	\end{equation}
	We deduce from \eqref{l329} and \eqref{l3210} that
	\begin{equation}\label{l3211}
	\limsup_{\epsilon\rightarrow 0} |\cl v_{\epsilon}|(D) \leq |\cl v|(D). 
	\end{equation}
	Combining \eqref{l3211} with \eqref{l326}, we obtain \eqref{l322}.
\end{proof}

Next we show that the above local approximation can be improved to global approximation, given that our domain $D$ is sufficiently smooth. We do this by showing that, for every $v\in V$, one can extend it to some $\tilde{v}\in L^2(\mb R^3;\mb R^2)$ such that $\text{curl}\tilde v$ is a finite Radon measure and $|\text{curl}\tilde v|(\partial D)= 0$. Such arguments follow the extension techniques for BV functions in Lipschitz domains. (See, e.g., \cite{AFP}, Chapter 3.) In the following, for an open set $U\subset\mb R^3$, we define
\begin{equation*}
V(U):=\{v\in L^2(U;\mb R^2): \cl v \in \mathcal{M}(U)\}.
\end{equation*}

\begin{lemma}\label{l6.2}
	There exists an extension operator $T: V(D)\rightarrow V(\mb R^3)$, such that, for all $v\in V(D)$, $Tv|_{D} = v$ and $|\cl Tv|(\partial D)=0$.
\end{lemma}

\begin{proof}
	
	The proof is divided into four steps.
	
	\emph{Step 1.} First, we make some simplifications. Since $\Omega\subset \mb R^2$ is a smooth domain, we can find finitely many open rectangles $\{R_i\}$ such that $\overline{\Omega}\subset\cup_i R_i$, and each $R_i$ satisfies either $R_i\subset\Omega$ or $R_i\cap\partial\Omega\ne\emptyset$. If $R_i\cap\partial\Omega\ne\emptyset$, by a rotation and translation, we may assume that $\partial\Omega\cap R_i$ is the graph of a smooth function defined on one side of $R_i$, and that it does not intersect with this side and the opposite side of $R_i$. Let $\{\zeta_i\}$ be a partition of unity subordinate to the covering $\{R_i\}$. Let $Q_i=R_i\times(0,L)$. We only need to define appropriate extensions $T_i$ on each $Q_i$ that satisfy the conclusions of Lemma \ref{l6.2}. Then we can define $Tv = \sum_i T_iv_i$ with $v_i=v \zeta_i$ in $\cup_i Q_i$, and extend $Tv$ to be zero in $\mb R^3\setminus\cup_iQ_i$. One can check that $Tv\in V(\mb R^3)$ satisfies the conclusions of Lemma \ref{l6.2}. Now we fix some $R_i$ with $R_i\cap \partial\Omega\ne\emptyset$. By a rotation and dilation, we may assume $R_i=l\times(-1,1)$ for some open interval $l\subset\mb R$. By a smooth deformation, we may assume that $R_i\cap\Omega=l\times(0,1)$. Such simplifications are standard in the BV setting. (See, e.g., \cite{AFP}, Chapter 3.) In the following, we define the extension $T_i$ on $Q_i$. We omit the subscript $i$ in the rest of the proof. We denote $Q^+ = Q\cap D$, $Q^- = Q\setminus\overline{Q^+}$, and $\Gamma = Q\cap\partial D$.
	
	\emph{Step 2.} Assume $v\in C^{\infty}(\overline{Q^+})$. We define $T$ to be a reflection across $\Gamma$:
	\begin{equation}\label{T}
	Tv(x_1,x_2,x_3) = 
	\begin{cases}
	\lt(v^1(x_1,x_2,x_3),v^2(x_1,x_2,x_3)\rt) & \text{ if } x\in \overline{Q^+},\\
	\lt(v^1(x_1,-x_2,x_3), -v^2(x_1,-x_2,x_3)\rt) & \text{ if } x\in Q^-.
	\end{cases}
	\end{equation}
	One can easily check that $Tv\in V(Q)$ and $|\cl Tv|(\Gamma)=0$. Moreover, we have $|\cl Tv|(Q)\leq 2|\cl v|(Q^+)$.
	
	\emph{Step 3.} Assume $v\in C^{\infty}(Q^+)$. We define $Tv$ as in \eqref{T} in $Q^-$. It is clear that $Tv\in L^2(Q)$. Let $Q_{\eps}^+=\{x\in Q:x_2>\eps\}$ and $\Gamma_{\eps}=\{x\in Q: x_2=\eps\}$. We define $v^{\eps}(x_1,x_2,x_3)=v(x_1,x_2+\eps,x_3)$. Let $\tau=(1,0,0)$ and $(v^{\eps})_{\tau}=v^{\eps}(x_1,0,x_3)\cdot\tau$ be the planar tangential component of $v$ on $\Gamma_{\eps}$, where we identity two-dimensional vector fields as three-dimensional vector fields with the $x_3$ component equal to zero. Let $\varphi\in C^1_c(Q)$ be a test function with $\sup|\varphi|\leq 1$. Using an integration by parts, and noting that $v$ has compact support in $Q^+$ with respect to the sides of $Q$, we have
	\begin{equation*}
	-\int_{Q^+} v\cdot\hat\nabla^{\perp}\varphi dx = -\lim_{\eps\rightarrow 0} \int_{Q^+_{\eps}}v\cdot\hat\nabla^{\perp}\varphi dx = \lim_{\eps\rightarrow 0}\left(\int_{Q^+_{\eps}}\cl v\varphi dx - \int_{\Gamma_{\eps}} (v\varphi)\cdot\tau d\mathcal{H}^2 \right).
	\end{equation*}
	Since $\cl v \in (L^1(Q^+))$, we have
	\begin{equation*}
	\lim_{\eps\rightarrow 0}\int_{Q^+_{\eps}}\cl v\varphi dx = \int_{Q^+}\cl v\varphi dx.
	\end{equation*}
	It follows that $\lim_{\eps\rightarrow 0}\int_{\Gamma_{\eps}} (v\varphi)\cdot\tau d\mathcal{H}^2$ exists and
	\begin{equation}\label{l627}
	\lim_{\eps\rightarrow 0}\int_{\Gamma_{\eps}} (v\varphi)\cdot\tau d\mathcal{H}^2= \int_{Q^+} v\cdot\hat\nabla^{\perp}\varphi dx+ \int_{Q^+}\cl v\varphi dx.
	\end{equation}
	Using an integration by parts, we have that
	\begin{equation}\label{l628}
	\int_{\Gamma}(v^{\eps}\varphi^{\eps} - v^{\eps}\varphi)\cdot \tau d\mathcal{H}^2= \int_{Q^+}(\cl v^{\eps}\varphi^{\eps}-\cl v^{\eps}\varphi) dx+ \int_{Q^+}v^{\eps}\cdot(\hat\nabla^{\perp}\varphi^{\eps}-\hat\nabla^{\perp}\varphi)dx. 
	\end{equation}
	Since $\varphi^{\eps}\rightarrow \varphi$ uniformly and $\cl v^{\eps}$ is uniformly bounded in $L^1(Q^+)$, we have
	\begin{equation}\label{l629}
	\int_{Q^+}(\cl v^{\eps}\varphi^{\eps}-\cl v^{\eps}\varphi)dx \rightarrow 0.
	\end{equation}
	A similar argument yields 
	\begin{equation}\label{l630}
	\int_{Q^+}v^{\eps}\cdot(\hat\nabla^{\perp}\varphi^{\eps}-\hat\nabla^{\perp}\varphi)dx \rightarrow 0.
	\end{equation}
	Therefore, we conclude from \eqref{l627}-\eqref{l630} that
	\begin{equation}\label{l631}
	\lim_{\eps\rightarrow 0}\int_{\Gamma} (v^{\eps}\varphi)\cdot \tau d\mathcal{H}^2= \int_{Q^+} v\cdot\hat\nabla^{\perp}\varphi dx + \int_{Q^+}\cl v\varphi dx.
	\end{equation}
	Using a similar argument in $Q^-$ and noting that $Tv|_{Q^-}$ is a reflection of $v$, we deduce that
	\begin{equation*}
	-\int_Q Tv\cdot\hat\nabla^{\perp}\varphi dx= \int_{Q^+} \cl v \varphi dx+ \int_{Q^-} \cl Tv \varphi dx,
	\end{equation*}
	which implies $\cl Tv\in\mathcal{M}(Q)$, and $|\cl Tv|(\Gamma) = 0$.
	
	\emph{Step 4.} Finally, assume $v\in V(Q^+)$. We define $Tv$ to be the reflection as in Step 3. Using Lemma \ref{l6.1}, we can find a sequence $\{v_k\}\subset V(Q^+)\cap C^{\infty}(Q^+)$ such that \eqref{l321} and \eqref{l322} hold. According to \eqref{l631}, we can define $(v_k)_{\tau}$ as the limit of $(v_k^{\eps})_{\tau}$ in the sense of weak convergence of measures. Let $\varphi$ be as in Step 3. For any $0<\eps<1$ and any $k$ and $l$, we have
	\begin{equation}\label{l633}
	\begin{split}
	\lt\lvert \int_{\Gamma}(v_k)_{\tau}\varphi d\haus^2-\int_{\Gamma}(v_l)_{\tau}\varphi d\haus^2 \rt\rvert &\leq \frac{1}{\eps}\int_{0}^{\eps}\lt|\int_{\Gamma}(v_k)_{\tau}\varphi d\haus^2-\int_{\Gamma}(v_k^t)_{\tau}\varphi^t d\haus^2 \rt|dt\\
	&+\frac{1}{\eps}\int_{0}^{\eps}\lt|\int_{\Gamma}(v_l)_{\tau}\varphi d\haus^2-\int_{\Gamma}(v_l^t)_{\tau}\varphi^t d\haus^2 \rt|dt\\
	&+\frac{1}{\eps}\int_{0}^{\eps}\lt|\int_{\Gamma}(v_k^t)_{\tau}\varphi^t d\haus^2-\int_{\Gamma}(v_l^t)_{\tau}\varphi^t d\haus^2 \rt|dt.
	\end{split}
	\end{equation} 
	By \eqref{l631}, we have
	\begin{equation}\label{l635}
	\begin{split}
	\bigg|\int_{\Gamma}(v_k)_{\tau}\varphi d\haus^2&-\int_{\Gamma}(v_k^t)_{\tau}\varphi^t d\haus^2 \bigg|\\
	&\leq \lt|\int_{Q^+\setminus Q^+_{t}}\cl v_k\varphi dx\rt| + \lt|\int_{Q^+\setminus Q^+_{t}}v_k\cdot\hat{\nabla}^{\perp}\varphi dx\rt|\\
	&\leq C\lt(\lVert \cl v_k\rVert_{L^1(Q^+\setminus Q^+_t)}+\lVert v_k\rVert_{L^2(Q^+\setminus Q^+_t)}\rt)\\
	&\leq C\lt(\lVert \cl v_k\rVert_{L^1(Q^+\setminus Q^+_{\eps})}+\lVert v_k\rVert_{L^2(Q^+\setminus Q^+_{\eps})}\rt).
	\end{split}
	\end{equation}
	Also we have
	\begin{equation}\label{l634}
	\frac{1}{\eps}\int_{0}^{\eps}\lt|\int_{\Gamma}(v_k^t)_{\tau}\varphi^t d\haus^2-\int_{\Gamma}(v_l^t)_{\tau}\varphi^t d\haus^2 \rt|dt \leq \frac{1}{\eps}\int_{Q^+\setminus Q^+_{\eps}}\lt|v_k-v_l \rt||\varphi|dx.
	\end{equation}
	For all $\eps$ sufficiently small, by taking the supremum over all test functions that are compactly supported in $Q^+\setminus \overline{Q^+_{\eps}}$ in \eqref{l329}, we obtain $$\limsup_{k\rightarrow\infty}\lVert \cl v_k\rVert_{L^1(Q^+\setminus Q^+_{\eps})}\leq |\cl v|(Q^+\setminus \overline{Q^+_{2\eps}}).$$ 
	Therefore, by first letting $k,l\rightarrow\infty$ and then letting $\eps\rightarrow 0$, we deduce from \eqref{l633}-\eqref{l634} that $\lim_{k\rightarrow \infty}\int_{\Gamma}(v_k)_{\tau}\varphi d\haus^2$
	exists. Define
	\begin{equation}\label{l638}
	\int_{\Gamma}(v)_{\tau}\varphi d\haus^2 = \lim_{k\rightarrow \infty}\int_{\Gamma}(v_k)_{\tau}\varphi d\haus^2.
	\end{equation}
	The above definition of $(v)_{\tau}$ does not depend on the sequence $\{v_k\}$. Using \eqref{l321}, \eqref{l322}, \eqref{l631} and \eqref{l638}, we have
	\begin{equation*}
	\begin{split}
	-\int_{Q^+} v\cdot\hat\nabla^{\perp}\varphi dx = -\lim_{k\rightarrow 0}\int_{Q^+} v_k\cdot&\hat\nabla^{\perp}\varphi dx= \lim_{k\rightarrow\infty}\lt(\int_{Q^+}\cl v_k\varphi dx-\int_{\Gamma}(v_k)_{\tau}\varphi d\haus^2\rt)\\
	& \leq \liminf_{k\rightarrow\infty}|\cl v_k|(Q^+)-\int_{\Gamma}(v)_{\tau}\varphi d\haus^2\\
	& =|\cl v|(Q^+)-\int_{\Gamma}(v)_{\tau}\varphi d\haus^2.
	\end{split}
	\end{equation*}
	Since $Tv$ in $Q^-$ is a reflection of $v$, the traces of $Tv|_{Q^+}$ and $Tv|_{Q^-}$ in the sense of \eqref{l638} agree on $\Gamma$. It follows that
	\begin{equation*}
	\begin{split}
	-\int_{Q} Tv\cdot\hat\nabla^{\perp}\varphi dx \leq 2|\cl v|(Q^+).
	\end{split}
	\end{equation*}
	Taking the supremum over all $\varphi$ with $\sup|\varphi|\leq 1$, we conclude that $\cl Tv\in \mathcal{M}(Q)$, and $|\cl Tv|(Q)\leq 2|\cl v|(Q^+)$. On the other hand, we have $|\cl Tv|(Q)=|\cl Tv|(Q^+)+|\cl Tv|(Q^-)+|\cl Tv|(\Gamma) \geq 2|\cl v|(Q^+)$. Therefore, it is clear that $|\cl Tv|(\Gamma)=0$. 
\end{proof}

Combining Lemmas \ref{l6.1} and \ref{l6.2}, we obtain Proposition \ref{prop2.2}.

\end{document}